\documentclass[10pt, a4paper, twoside]{article}
\usepackage{a4}

\usepackage{amsfonts}
\usepackage{amsmath}
\usepackage{amssymb}
\usepackage{mathrsfs}
\usepackage{vmargin}

\long\def\symbolfootnote[#1]#2{\begingroup%
\def\thefootnote{\fnsymbol{footnote}}\footnote[#1]{#2}\endgroup}

\setmarginsrb{20mm}{20mm}{20mm}{20mm}{10mm}{10mm}{10mm}{10mm}

\RequirePackage{amsthm}

\theoremstyle{plain}
\newtheorem{theorem}{Theorem}
\newtheorem{lem}{Lemma}
\newtheorem{cor}{Corollary}
\newtheorem{observ}{Observation}

\theoremstyle{definition}
\newtheorem{definition}{Definition}
\newtheorem{rem}{Remark}
\newtheorem{ex}{Example}

\newcommand{\Om}{\Omega}
\newcommand{\dist}{{\rm dist}}
\newcommand{\bd}{\partial}
\newcommand{\dM}{d_{\rm{inn}}}

\newcommand{\distM}{{\rm dist}_{\rm{inn}}}
\newcommand{\distMa}{{\rm dist}_{\rm{M}}}
\newcommand{\distMst}{{\rm dist}^{*}_M}
\newcommand{\diam}{{\rm diam}}
\newcommand{\calF}{\mathcal{F}}
\newcommand{\R}{\mathbb{R}}
\newcommand{\N}{{\mathbb N}}
\newcommand{\Rn}{{\mathbb R}^n}
\newcommand{\ga}{\gamma}

\newcommand{\fv}{f^{-1}}
\newcommand{\clOmP}{{\overline{\Om}\mspace{1mu}}^P}
\newcommand{\ovOm}{\overline{\Om}}
\newcommand{\ovD}{\overline{D}}
\newcommand{\bdySP}{\partial_{\rm SP}}
\newcommand{\bdyP}{\partial_{\rm P}}
\newcommand{\bdyPF}{\partial_{\rm PF}}
\newcommand{\modq}{{\rm Mod}_Q}
\newcommand{\modp}{{\rm Mod}_p}
\newcommand{\modn}{{\rm Mod}_n}
\newcommand{\Mod}{{\rm Mod}}

\newcommand{\kom}[1]{}
\renewcommand{\kom}[1]{{\bf [#1]}}

\begin{document}

\title {Prime ends in metric spaces and boundary extensions of mappings}

\author{
Tomasz Adamowicz{\small{$^1$}}
\\
\it\small Institute of Mathematics, Polish Academy of Sciences \\
\it\small ul. \'Sniadeckich 8, 00-656 Warsaw, Poland\/{\rm ;}
\it\small T.Adamowicz@impan.pl
}

\date{}
\maketitle

\footnotetext[1]{T. Adamowicz was supported by a grant Iuventus Plus of the Ministry of Science and Higher Education of the Republic of Poland, Nr 0009/IP3/2015/73.}


\begin{abstract}
 By using the inner diameter distance condition we define and investigate new, in such a generality, class $\calF$ of homeomorphisms between domains in metric spaces and show that, under additional assumptions on domains, $\calF$ contains (quasi)conformal, bi-Lipschitz and quasisymmetric mappings as illustrated by examples. Moreover, we employ a prime ends theory in metric spaces and provide conditions allowing continuous and homeomorphic extensions of mappings in $\calF$  to topological closures of domains, as well as homeomorphic extensions to the prime end boundary. Domains satisfying the bounded turning condition, locally and finitely connected at the boundary and the structure of prime end boundaries for such domains play a crucial role in our investigations.

 We apply our results to show the Koebe theorem on arcwise limits for mappings in $\calF$. Furthermore, relations between the Royden boundary and the prime end boundary are presented.

 Our work generalizes results due to Carath\'eodory, N\"akki, V\"ais\"al\"a and Zori\v c.
\newline
\newline \emph{Keywords}: bi-Lipschitz, bounded turning, collared, extension, finitely connected at the boundary, linearly connected, homeomorphism, inner diameter distance, Koebe, Mazurkiewicz, modulus of curve families, prime end, quasiconformal, quasisymmetric, Royden
\newline
\newline
\emph{Mathematics Subject Classification (2010):} Primary: 30D40; Secondary: 30L10, 31B25, 30C65.
\end{abstract}

\section{Introduction}

 The extension problem for mappings between two open domains has been studied in various settings and for various kinds of extension properties. The classical setting, from which our studies are originated, includes domains in the Euclidean spaces and conformal mappings. In 1913 Carath\'eodory~\cite{car1} and Osgood--Taylor~\cite{ot} proved  independently that a conformal mapping between Jordan domains extends to a homeomorphism of the closures of domains. Counterexamples, such as a slit-disc, show that in general a homeomorphic extension need not be possible. However, Carath\'eodory~\cite{car1} created and studied an abstract type of a boundary, the so-called prime end boundary and proved that for a planar simply-connected domains a homeomorphic extension of a conformal map is possible with respect to the prime end closure of the target domain. Subsequently, prime ends and their properties have been studied in more general domains and in higher dimensions, see Section~\ref{sec-pet} for further references. Moreover, prime ends have been employed to investigate other topics, e.g. the theory of continua, see Carmona--Pommerenke~\cite{cp1}, local connectivity, see Rempe~\cite{rem}, the dynamical systems, see Koropecki--Le Calvez--Nassiri~\cite{kcn}, the boundary behavior of solutions to elliptic PDEs, see Ancona~\cite{An} and the studies of the $p$-harmonic Dirichlet problem in metric spaces, see Estep--Shanmugalingam~\cite{es}.

 Another direction of related studies arises from quasiconformal mappings (qc-mappings for short). In this setting the higher-dimensional counterpart of the Carath\'eodory--Osgood--Taylor theorem for Jordan domains fails to exist as showed by Kuusalo~\cite{kus}. Nevertheless, for Euclidean domains locally connected at the boundary and quasiconformally collared the homeomorphic extension of a qc-mapping exists, see V\"ais\"al\"a~\cite[Section 17]{va1} and Gehring~\cite{geh}; see also Gehring--Martio~\cite{gem}, Herron--Koskela~\cite{hek} and N\"akki~\cite{na1, na2} for further boundary properties of domains implying the homeomorphic and continuous extension properties. Moreover, results concerning homeomorphic extension of a qc-mapping with respect to prime end closures of a target domain are due to N\"akki~\cite{na} and V\"ais\"al\"a~\cite{va2}.

 Let us also mention that the extension problems have been investigated in the context of other generalizations of (quasi)conformal mappings, such as $Q$-homeomorphisms, see e.g. Chapter 13 in Martio--Ryazanov--Srebro--Yakubov~\cite{mrsy}.

 The extension results presented above are studied largely for Euclidean domains. The main goal of our work and the key novelty is to study extension properties for a class of homeomorphisms between domains in metric spaces defined via the internal diameter distance, denoted by $\distM$, see Definition~\ref{def-dM}. We extract a condition essential for continuous and homeomorphic extensions to the topological closures and closures with respect to the prime end boundary. Namely, we say that a homeomorphism $f:\Om \to\Om'$ belongs to a class $\calF(\Om, \Om')$ if the following condition holds for any two connected sets $E, F\subsetneq \Om$:
\[
  \distM(E, F) =0 \quad \hbox{if and only if} \quad \distM(f(E), f(F))=0,
\]
 where $\Om\subset X$ and $\Om'\subset Y$ are bounded domains in complete doubling quasiconvex metric spaces $X$ and $Y$. See Definition~\ref{class-map} and Section~\ref{sect-map-class} for details and further remarks. According to our best knowledge, class $\calF$ has not been studied before in such a general setting. We illustrate the above definition in Examples~\ref{ex1}-\ref{ex-qs}, where we show that, under additional assumptions on domains, conformal, quasiconformal, bi-Lipschitz and quasisymmetric mappings belong to class $\calF$. Among the boundary properties of domains employed in our studies let us mention the local and finite connectedness at the boundary, also collardness, qc-flatness, uniformity and the bounded turning condition, see Preliminaries.

 The essential role in our studies is played by a prime ends theory for domains in metric spaces developed recently in \cite{abbs, es}, as extension conditions are stated in terms of prime ends. We recall basic definitions for prime ends in Section~\ref{sec-pet}. Further properties of prime ends such as the structure of the prime end boundary for domains finitely (locally) connected at the boundary are recalled in Section~\ref{sect-mainr} where they are applied, see Theorems~\ref{thm-fin-con-homeo}-\ref{thm-clOmm-cpt-new} and Corollary~\ref{cor-loc-conn}. In the same section we prove the key lemma of the paper, Lemma~\ref{lem-maps}. The result shows that mappings in class $\calF$ map chains of sets to chains and preserve equivalency of chains, giving rise to the corresponding properties for (prime) ends. Moreover, for a mapping in $\calF$ and a target domain finitely connected at the boundary it holds that a preimage of a singleton prime end is a singleton, that is both prime ends have singleton impressions (see Definition~\ref{def-chain}\eqref{impr}). Using Lemma~\ref{lem-maps} we establish the following three types of extension results for homeomorphisms in $\calF$:
\begin{itemize}
\item[(1)] An extension to a continuous map between closures of domains
 (Theorem~\ref{Thm1-homeo-ext}).

 In particular, if a domain $\Om$ satisfies the bounded turning condition and a domain $\Om'$ is finitely connected at the boundary, then $f\in \calF(\Om, \Om')$ extends continuously to a mapping $F: \overline{\Om}\to \overline{\Om'}$, cf. Corollary~\ref{cor-homeo-ext}.

\item[(2)] An extension to a homeomorphism between closures of domains (Theorem~\ref{Thm1-homeo-ext-nec}).
\item[(3)] An extension to a homeomorphism between the topological and the prime end boundaries (Theorem~\ref{Thm1-homeo-pext}).

    As corollaries we retrieve results for qc-mappings: a counterpart of Theorem 4.2 in N\"akki~\cite{na} in collared domains, see Corollary~\ref{Thm1-quasiconf-pext} and an extension result for target domains finitely connected at the boundary due to V\"ais\"al\"a~\cite[Section 3.1]{va2}, see Corollary~\ref{Thm1-quasiconf-pext-finit}.
\end{itemize}

Last section is devoted to studies of two applications of extension results: in Section~\ref{sect-koebe} we present a variant of the Koebe theorem, see Theorem~\ref{thm-koebe}. Namely, we show that mappings in class $\calF$ between domains locally and finitely connected at the boundary, respectively, have arcwise limits along end-cuts.

Section~\ref{sect-royden} brings on stage another type of abstract boundary, the Royden boundary. Upon stating the necessary definitions we recall an extension criterion for quasiconformal mappings expressed in terms of fibers, see Theorem~\ref{thm51-sod}. Then we show that for Euclidean domains finitely connected at the boundary fibers over a given boundary point correspond to prime ends with impressions at this point, see Theorem~\ref{thm-Royden-prime}. As a corollary, we obtain an upper estimate for a number of components of fibers in the Royden compactification for John domains, see Corollary~\ref{cor-royden-john}.

\section{Preliminaries}

In what follows by $(X, d_X)$ we denote a metric space $X$ with a distance function $d_X$, while $\Om\subset X$ is an open connected subset of $X$. By a curve in $X$ we understand a continuous mapping $\ga:[a, b]\to X$. The image of $\ga$ (loci) is defined as $|\ga|=\ga([a,b])$. The length of $\ga$ is denoted by $l(\ga)$ and we say that $\ga$ is rectifiable if $l(\ga)<\infty$. Every rectifiable curve admits the so-called arc-length parametrization, see e.g. Theorem 3.2 in Haj\l asz~\cite{haj}.

Let $\gamma$ be a curve in $\Om$. We define its diameter as follows:
\begin{equation*}
 \diam\,\gamma:= \sup d_X(x,y),
\end{equation*}
where the supremum is taken over all points $x,y\in \gamma$.

The following metrics will play an important role in the paper, as they are used in the definition of chain and (prime) ends (see Definitions~\ref{def-chain} and \ref{def-end}) and in the definition of the main class of mappings studied in the paper, see Definition~\ref{class-map}.

\begin{definition}\label{def-dM}
We define the \emph{inner diameter distance} $\dM$ on $\Om$ by
\[
     \dM(x,y) =\inf \diam\,\gamma,
\]
where the infimum is taken over all locally rectifiable curves $\gamma$ joining $x,y \in \Om$ such that $\gamma\subset \Om$.
\end{definition}
The definition of $\dM$ naturally extends to the distance between two sets $E, F\subset \Om$ denoted $\distM(E, F)$.

 The above metric is more commonly called \emph{relative diameter distance}, sometimes known as \emph{in\-ternal/in\-trin\-sic diameter distance}. Clearly, $\dM$ is a metric on $\Om$. One can also define the similar distance function, the so-called \emph{Mazurkiewicz metric} between points $x$ and $y$ by taking the infimum of diameters of all connected sets in $\Om$ containing $x,y$, cf. Definition 8.10 in~\cite{abbs}. Let us denote such metric by $d_{M}$. See also Bj\"orn--Bj\"orn--Shanmugalingam~\cite{bbs1, bbs2} for further studies on Mazurkiewicz distance and its role in the geometry of sets and nonlinear potential theory on metric spaces, see also Mazurkiewicz~\cite{maz}. The following relations between the aforementioned metrics hold for all $x, y\in \Om$:
\begin{equation}\label{metric-rel}
\dM(x,y)\ge d_{M}(x,y) \ge d_{X}(x,y).
\end{equation}
In particular, if $\dM(x, y)=0$, then $d_{M}(x,y)=0$ and $d_X(x,y)=0$.

\begin{definition}[Linearly connected space]\label{lin-con}
We say that a complete metric space $(X, d_X)$ is \emph{$L$-linearly connected} for some $L\geq 1$, if for all $x, y\in X$ there exists a continuum $K$ containing $x$ and $y$ such that
\[
 \diam K\leq L d(x,y).
\]
\end{definition}

Let us comment on the above definition.

\begin{rem}\label{rem-lin-con}$\phantom{AAA}$

\begin{enumerate}
\item The $L$-linearly connected space is also known as the space satisfying the \emph{$L$-bounded turning condition} or \emph{LLC(1)}, cf. e.g. Hakobyan--Herron~\cite[Sections 2.B, 4.B]{hakher} and MacKay~\cite{mack}.
\item It turns out that in Definition~\ref{lin-con} one may assume that $K$ is an arc, at the cost of possibly increasing $L$ by an arbitrarily small amount, see e.g.~\cite{mack}.
\item A linearly connected domain is locally connected and locally connected at every point of its boundary, see e.g. discussion in~\cite{hakher}.
 \end{enumerate}
\end{rem}

\medskip

\noindent{\bf Boundary of domains in metric spaces}

\smallskip
In what follows we will study the boundary behavior of mappings. For this reason, we gather the necessary  definitions of various types of boundary points and related domains. Since definitions and examples below rely or are related to the notion of the curve modulus, we will first recall it, see e.g. V\"ais\"al\"a~\cite[Chapter 6]{va1} for discussion in the Euclidean setting and Heinonen--Koskela~\cite{hk} and Heinonen~\cite[Chapter 7]{hei01} for the definitions and applications of the modulus in metric measure spaces.

Let $\Gamma$ be a family of curves in a domain $\Om\subset X$ and let $1\leq p <\infty$. Then the \emph{$p$-modulus of curve family $\Gamma$} is defined as follows:
\[
 \modp \Gamma:=\inf_{\varrho\in F(\Gamma)}\int_{X}\varrho^p d\mu,
\]
where $\mu$ is a Borel regular measure in metric space $X$, whereas $F(\Gamma)$ stands for the set of admissible functions. Namely, a nonnegative Borel function $\varrho:X \to[0,\infty]$ is \emph{admissible for $\Gamma$} if
\[
 \int_{\gamma} \varrho ds \geq 1,
\]
for every locally rectifiable $\gamma\in \Gamma$. If $F(\Gamma)$ is empty, then by convention we define $\modp \Gamma =\infty$. This happens if $\Gamma$ contains a constant curve. Among fundamental properties of the $p$-modulus we mention that it defines an outer measure on the set of all curves in $X$.

Let $\Omega \subset X$ be a domain and let $E, F \subset \Omega$. By $\modp(E, F, \Om)$ we denote the modulus of the curve family $\Gamma(E, F, \Omega)$ consisting of all rectifiable curves $\ga$ in $\Omega$ which join $E$ and $F$, i.e. one of the endpoints of $\ga$ belongs to $E$, the other to $F$ and $\ga\setminus(E\cup F)\subset \Om$.

\begin{definition}\label{qc-flat}
We say that a domain $\Omega$ in a metric space $(X, d_X)$ is \emph{quasiconformally flat}(\emph{QC-flat} for short) at $x\in \bd \Om$ if for any pair of connected subsets $E, F \subseteq \Omega$ such that $ x\in \overline{E} \cap \overline{F}$ we have $\modq (E,F,\Omega)=\infty$.
\end{definition}

In order to illustrate the definition let us mention that a locally $Q$-Loewner uniform domain is QC-flat at every boundary point as follows from Fact 2.12 in Herron~\cite{herCam}.

Next, we define a class of collared domains which plays an important role in the studies of boundary behavior of quasiconformal mappings, see V\"ais\"al\"a~\cite[Chapter 17]{va1}, N\"akki~\cite{na} in the Euclidean setting and
\cite[Section 3.3]{aw2} for the studies in Heisenberg group $\mathbb{H}_1$.

\begin{definition}\label{def-collar}
 Let $\Om\subset \Rn$ be a domain and let $x\in \bd \Om$. We say that $\Om$ is \emph{quasiconformally collared at $x$} (\emph{collared} for short) if there exists a neighborhood $U$ of $x$ (in $\Rn$) and a homeomorphism $g$ from $U\cap \overline{\Om}$ onto $\{x\in \Rn: |x|<1, x_n\geq 0\}$ such that $g$ restricted to $U\cap \Om$ is quasiconformal.

Notice that by Topology $g$ maps $U\cap \bd \Om$ to a $(n-1)$-dimensional ball in $\Rn$.
\end{definition}

The following class of domains is crucial from the point of view of extension results studied below, as well as, from the perspective of prime ends theory, as it turns out that prime end boundaries have particularly simple structure in such domains, see Theorem~\ref{thm-fin-con-homeo} below and discussion following it, see also Sections 10 and 11 in \cite{abbs}.

\begin{definition}\label{def-fin-con}
We say that $\Omega\subset X$ is \emph{finitely connected at a point} $x \in \bd \Om$ if for every $r>0$ there exists a bounded open set $U$ in $X$ containing $x$ such that $x\in U\subset B(x, r)$ and $U \cap \Omega$ has has only finitely many components. If $\Omega$ is finitely connected at every boundary point, then we say it is finitely connected at the boundary.

In particular, if $U\cap \Om$ has exactly one component, then we say that $\Om$ is \emph{locally connected at} $x\in \bd \Om$.
\end{definition}

If a domain $\Om$ satisfies one of the conditions of Definition~\ref{def-fin-con} at every boundary point, then we say that $\Om$ is, respectively, \emph{finitely (locally) connected at the boundary}.

\begin{ex}[domains finitely and locally connected at the boundary]\label{ex-domains}
 The following domains are finitely connected at the boundary:
 \begin{itemize}
 \item[(1)] John domains in complete metric spaces, see \cite[Theorem 11.3]{abbs}
 \item[(2)] Weakly linearly locally connected domains in $\R^n$ (WLLC domains), see Herron--Koskela~\cite[Section 2]{hek}
 \item[(3)] Sobolev capacity domains (SC domains), see~\cite[Section 2]{hek}
 \item[(4)] Almost John domains, see Definition 11.4 and Theorem 11.5 in \cite{abbs}
 \item[(5)] Collared domains, see Definition~\ref{def-collar} and V\"ais\"al\"a~\cite[Theorem 17.10]{va1} in Euclidean setting; see also Definition 3.9 and Observation 3.2 in \cite{aw2} in the setting of $\mathbb{H}_1$.
\end{itemize}

 The following domains are locally connected at the boundary:
 \begin{itemize}
 \item[(1)] Uniform domains in complete metric spaces, see \cite[Proposition 11.2]{abbs}
 \item[(2)] Linearly locally connected domains in $\R^n$ (LLC domains), see~\cite[Section 2]{hek}
 \item[(3)] Quasiextremal distance domains (QED), see~\cite[Section 2]{hek}
 \item[(4)] Linearly connected domains, see Hakobyan--Herron~\cite[Section 2.C]{hakher}
 \item[(5)] Jordan domains in $\R^n$, see \cite[Definition 17.19, Theorem 17.20]{va1}.
 \end{itemize}
\end{ex}

\noindent{\bf Quasiconformal mappings and their counterparts in metric spaces}

\smallskip
Our studies for mappings largely grow from the similar studies for (quasi)conformal mappings. For the rudimentary properties of quasiconformal mappings we refer to e.g. \cite{va1} in Euclidean setting and \cite{hk} in the metric setting. Relations between several definitions of quasiconformal mappings in metric measure spaces, including conditions implying their equivalence, are presented e.g. in Koskela--Wildrick~\cite{kw} and Williams~\cite{wil}. Below we employ the following definition.

\begin{definition}\label{def-qc}
 Let $(X, d_X)$ be a $Q$-regular metric measure space $(Q>1)$ and $\Om, \Om'\subset X$ be domains (not necessarily bounded). A homeomorphism $f: \Om \to \Om'$ is called \emph{$K_f$-quasiconformal in $\Om$} if there exists a constant $K_f\geq 1$ such that for any family of curves $\Gamma$ in $\Om$ we have
\begin{equation}\label{def-qc-mod}
 \frac{1}{K_f}\modq(\Gamma)\leq  \modq(f\Gamma ) \leq K_f \modq(\Gamma).
\end{equation}
If $X=\R^n$ and $K_f=1$, then $f$ is conformal.
\end{definition}

In next definition we present another generalization of quasiconformal mappings in metric spaces, the so-called quasisymmetric mappings. The concept was introduced by Ahlfors--Buerling~\cite{ahlb} in the context of the boundary behavior of planar quasiconformal mappings and in the metric spaces by Tukia--V\"ais\"al\"a~\cite{tuv}. We remark that among examples of quasisymmetric mappings there are bi-Lipschitz mappings, and quasiconformal mappings between domains in $\R^n$ can be characterized as locally quasisymmetric, see Heinonen~\cite[Theorems 11.14 and 11.19]{hei01} and Section 4 in Heinonen--Koskela~\cite{hk}.

\begin{definition}\label{def-qsym}
 Let $(X, d_X)$ and $(Y, d_Y)$ be metric spaces. A homeomorphism $f: X \to Y$ is called \emph{$\eta$-quasisymmeric} if there exists a homeomorphism $\eta:[0,\infty)\to [0,\infty)$ such that the following condition holds for all triples $a,b, x$ of distinct points in $X$:
\begin{equation*}\label{def-qsym-cond}
 \frac{d_Y(f(x), f(a))}{d_Y(f(x), f(b))}\leq \eta\left(\frac{d_X(x, a)}{d_X(x, b)}\right).
\end{equation*}
\end{definition}

We refer to \cite{hei01} and \cite{hk} for further properties of quasisymmetric mappings and to~\cite{kw} for a survey of relations between quasisymmetricity and various definitions of quasiconformal maps (in particular, see Theorem 4.3 in \cite{kw}).

\begin{definition}\label{def-bilip}
 Let $(X, d_X)$ and $(Y, d_Y)$ be metric spaces and $\Om \subset X$, $\Om' \subset Y$ be domains. A map $f: \Om \to \Om'$ is called \emph{$L$-bi-Lipschitz} if there exists a constant $L$ such that the following condition holds
\begin{equation*}\label{def-bilip-ineq}
\frac{1}{L}d_X(x,y)\leq d_Y(f(x), f(y))\leq Ld_X(x,y)\quad  \hbox{for any } x, y \in \Om.
\end{equation*}
\end{definition}

In order to see these mappings in the wider context, let us recall that in Euclidean setting, an orientation preserving bi-Lipschitz map is a map of bounded length distortion (the so-called BLD-map), and thus quasiregular.
See Chapter 14.78 in Heinonen--Kilpel\"ainen--Martio~\cite{hkm} and further references therein. Furthermore, in the setting of metric spaces, one observes that an $L$-bi-Lipschitz map is $\eta$-quasisymmetric with $\eta(t)=L^2t$ (cf. Definition~\ref{def-qsym}), see e.g. Chapter 10 in Heinonen~\cite{hei01}.

\section{Prime ends in metric spaces}\label{sec-pet}

The first theory of prime ends is due to Carath\'eodory~\cite{car1} who studied prime ends in
simply-connected domains in the plane. For a comprehensive introduction to Carath\'eodory's prime ends we refer to Chapter 9 of a book by Collingwood--Lohwater~\cite{cl}. Subsequently, the theory has been developing to include more general domains in the plane and in higher dimensional Euclidean spaces, e.g. Freudenthal~\cite{Fre}, Kaufman~\cite{Kau}, Mazurkiewicz~\cite{maz2}, and more recently Epstein~\cite{Ep} and N\"akki~\cite{na}. The latter one defines prime ends based on the notion of modulus, and therefore, suitable to investigate the quasiconformal mappings in $\R^n$ (cf. \cite{aw2} for a related to \cite{na} prime ends theory in the Heisenberg group $\mathbb{H}_1$).

In this paper we study the following types of ends and prime ends in a more general setting of metric spaces
recently proposed in~\cite{abbs, es}. In order to introduce basic definitions, it is enough to assume that $(X, d_X)$ is a complete doubling metric space. The construction of (prime) ends consists of number of auxiliary definitions. First, we define acceptable sets.

Let $\Om\subsetneq X$ be a bounded domain in $X$, i.e. a bounded nonempty connected open subset of $X$ that is not the whole space $X$ itself.

\begin{definition}\label{def-accset}
 We call a bounded connected set $E\subsetneq\Om$ an \emph{acceptable} set if $\overline{E}\cap \partial \Omega\not=\emptyset$.
\end{definition}

By discussion in \cite{abbs}, we know that boundedness and connectedness of an acceptable set $E$ implies that
$\overline{E}$ is compact and connected. Furthermore, $E$ is infinite, as otherwise we would have $\overline{E}=E \subset \Om$. Therefore, $\overline{E}$ is a continuum. Recall that a \emph{continuum} is a connected compact set containing at least two points.

\begin{definition}\label{def-chain}
We call a sequence $\{E_k\}_{k=1}^\infty$ of acceptable sets a \emph{chain} if it satisfies the following conditions:
\begin{enumerate}
\item \label{it-subset}
$E_{k+1}\subset E_k$ for all $k=1,2,\ldots$,
\item \label{pos-dist}
$\distMa(\Omega\cap\bd E_{k+1},\Omega\cap \bd E_k )>0$
for all $k=1,2,\ldots$,
\item \label{impr}
The \emph{impression} $\bigcap_{k=1}^\infty \overline{E}_k \subset \bd\Om$.
\end{enumerate}
\end{definition}

We remark that the impression is either a point or a continuum, since $\{\overline{E}_k\}_{k=1}^\infty$ is a decreasing sequence of continua. Furthermore, Properties \ref{it-subset} and \ref{pos-dist} above imply that $E_{k+1}\subset {\rm Int} E_{k}$. In particular, ${\rm Int} E_{k} \ne \emptyset$.

Our definition of a chain differs from that in Definition 4.2 in \cite{abbs} as now we use the Mazurkiewicz distance instead of the underlying metric. Such a modification is convenient since class $\calF$ of  mappings is defined in terms of the related inner diameter distance, also due to relations \eqref{metric-rel}. Furthermore, in what follows we will often study paths and their behavior under homeomorphisms in class $\calF$, also we will appeal to constructions involving paths and continua containing them. Let us also emphasize that such a modification does not affect results from \cite{abbs}, cf. Definition 2.3 in Estep--Shanmugalingam~\cite{es} and the discussion following it for comparison between the above definition and Definition 4.2 in \cite{abbs}. In general, there are more chains and ends in the sense of the above definition than in \cite{abbs} (see Definition~\ref{def-end} below), and thus a priori a prime end in the setting of \cite{abbs} need not be prime in our sense. We further note, that results in \cite{abbs} employed below,  which use the analog of condition~\eqref{pos-dist} in Definition~\ref{def-chain} for $d$ instead of $\distMa$ are, in fact, based on the positivity of the Mazurkiewicz distance. Nevertheless, upon appealing to results from \cite{abbs}, we comment about consequences of the difference between definitions here and in \cite{abbs}.

\begin{definition}\label{def-end}
(1) We say that a chain $\{E_k\}_{k=1}^\infty$ \emph{divides} the chain
$\{F_k\}_{k=1}^\infty$ if for each $k$ there exists $l_k$
such that $E_{l_k}\subset F_k$ (for all $l\geq l_k$).

(2) Two chains are \emph{equivalent} if they divide each other.

(3) A collection of all mutually equivalent chains is called an \emph{end} and denoted $[E_k]$, where
$\{E_k\}_{k=1}^\infty$ is any of the chains in the equivalence class.

(4) The \emph{impression of} $[E_k]$, denoted $I[E_k]$, is defined as the impression of
any representative chain.

The collection of all ends is called the \emph{end boundary} and is
denoted $\partial_E\Omega$.
\end{definition}

The impression of an end is independent of the choice of representative chain, see \cite[Section 4]{abbs}. Note also that if a chain $\{F_k\}_{k=1}^\infty$ divides $\{E_k\}_{k=1}^\infty$, then it divides every chain equivalent to $\{E_k\}_{k=1}^\infty$. Furthermore, if $\{F_k\}_{k=1}^\infty$ divides $\{E_k\}_{k=1}^\infty$,
then every chain equivalent to $\{F_k\}_{k=1}^\infty$ also divides $\{E_k\}_{k=1}^\infty$.
Therefore, the relation of division extends in a natural way from chains to ends, defining a partial order on ends.

\begin{definition}\label{prime-end}
We say that an end $[E_k]$ is a \emph{prime end} if it is not divisible by any other end. The collection of all prime ends is called the \emph{prime end boundary} and is denoted $\partial_P\Omega$.

Similarly, the collection of all prime ends with singleton impressions is denoted $\bdySP \Omega$.
\end{definition}

For the convenience of readers further results describing accessible points and the structure of prime ends in domains (locally) finitely connected at the boundary are presented in Section~\ref{sect-mainr}, where they are applied.

\section{The class $\calF$ of mappings}\label{sect-map-class}

We say that a metric space $(X, d_X)$ satisfies the set of assumptions (A) if the following hold.
\vspace{0.2cm}

{\bf Assumptions (A)}

\begin{itemize}
\item[(1)] $(X, d_X)$ is a complete doubling metric space.
\item[(2)] $X$ is quasiconvex. 
\end{itemize}

The first assumption on space $X$ implies, among other properties, that $X$ is proper, i.e. closed and bounded sets in $X$ are compact. The second allows us to infer that if $\Om\subset X$ is an open connected set, then $\Om$ is rectifiably connected and, thus path-connected, see Lemmas 4.37 and 4.38 in Bj\"orn--Bj\"orn~\cite{bb}.


The following class of mappings is the central subject of our studies.
\begin{definition}\label{class-map}
 Let $(X, d_X)$ and $(Y, d_Y)$ be metric spaces both satisfying assumptions (A). We say that a homeomorphism $f$ from a bounded domain $\Om\subset X$ onto a bounded domain $\Om'\subset Y$ belongs to class $\calF(\Om, \Om')$ if for any two connected sets $E, F \subsetneq \Om$ it holds that
 \begin{equation}\label{class-map-prop}
  \distM(E, F) =0 \quad \hbox{if and only if} \quad \distM(f(E), f(F))=0. \tag{$\mathcal{F}$}
 \end{equation}
\end{definition}

In the definition above we provide a minimal set of assumptions on the underlying domains. However, in further studies and examples illustrating this class of mappings we need to impose various additional conditions on the geometry of domains. This is the case, for instance, when one wants to check when all quasiconformal mappings between two given domains belong to class $\calF$, see examples below and Remark~\ref{rem-herkos}.

 Similar class of mappings was introduced and studied in the Euclidean setting in the context of extension properties of homeomorphisms by Zori\v c~\cite{zo2} (with the inner distance metric called in \cite{zo2}, the Mazurkiewicz metric). However, in \cite{zo2} one studies only the case of homeomorphisms between a ball and its image under given homeomorphism. Moreover, the results of \cite{zo2} heavily rely on the discussion of similar extension results for quasiconformal mappings between a Euclidean ball and its image.

Class $\calF$ contains several well-known classes of mappings as shown by the following examples. We present them in both Euclidean and metric settings and use various techniques to determine when a homeomorphism belongs to $\calF$.

\begin{ex}\label{ex1}$\phantom{AAAAAAAAAAAAAAAAAAAAAAAAAAAAAAAAAAaAAAAAAAAAAaAAAAAAAAAAAAAAAAAAAA}$
{\bf (a)}\,((Quasi)conformal mappings I) Let $B^n\subset \R^n$ be a ball and $f$ be (quasi)conformal mapping of $B^n$ onto a collared domain $\Om'\subset \R^n$. Then $f\in \calF(B^n, \Om')$.

 Indeed, let $E, F\subset B^n$ be connected such that $\distM(E,F)=0$. Then by the above discussion on metric $\dM$ it holds that also $d(E, F)=0$. Thus, there exists $x\in \overline{E}\cap \overline{F}$ and we need to consider two cases: $x\in {\rm Int}\, B^n$ and $x\in \bd B^n$.

 In the first case $\modn(E, F, B^n)=\infty$, as $\Gamma(E, F, B^n)$ contains a constant curve. Condition \eqref{def-qc-mod} in Definition~\ref{def-qc} implies that $\modn(f(E), f(F), \Om')=\infty$ which together with the injectivity of $f$ in turn again imply that $\Gamma(f(E), f(F), \Om')$ contains a constant curve. Hence, $\distM(f(E), f(F))=0$. The similar reasoning gives us implication in the opposite direction.

 On the other hand, if $x\in \bd B^n$, then since $B^n$ is QC-flat, we have by Definition~\ref{qc-flat} that
 $\modn(E, F, B^n)=\infty$. By quasiconformality of $f$ we have that $\modn(f(E), f(F), \Om')=\infty$. Moreover, since $\Om'$ is collared (see Definition~\ref{def-collar}), then Part (1) of Lemma 2.3 in N\"akki~\cite{na} implies that $d(f(E), f(F))=0$ and hence, $\distM(f(E), f(F))=0$.

  \vspace{0.2cm}

\noindent{\bf (b)}\,(bi-Lipschitz mappings) Let $f$ be a bi-Lipschitz mapping between domains $\Om$ and $\Om'$ in metric spaces such that $\Om$ and $\Om'$ are linearly connected (cf. Definitions~\ref{lin-con} and~\ref{def-bilip}). Then $f\in \calF(\Om, \Om')$.

The proof is analogous to the one in part a). Using the above notation, linear connectedness of $\Om'$ is used in order to infer from $d_Y(f(E), f(F))=0$ that $\distM(f(E), f(F))=0$. Similarly, linear connectedness of $\Om$ is employed to show that $d_X(E, F)=0$ implies that $\distM(E, F)=0$.
\end{ex}

It turns out that the above example can be generalized to include the case of bounded target domain $\Om'\subset \R^n$ finitely connected at the boundary. We present it separately, since the argument differs from the above.

\begin{ex}\label{ex-qc-finit}
Let $\Om=B^n\subset \R^n$ and $\Om'\subset \R^n$ be a bounded domain finitely connected at the boundary. Then, every quasiconformal map $f\in \calF(B, \Om')$.

In order to show this let us assume, that with the notation of Example~\ref{ex1}(a), it holds that $\distM(E,F)=0$. If \[
\overline{E}\cap \overline{F}\cap \Om \not=\emptyset,
\]
then, as above, by injectivity of $f$ we immediately obtain that $\distM(f(E), f(F))=0$. Let us, therefore, assume that  $\overline{E}$ and $\overline{F}$ intersect only at the boundary of $\Om$. Then, as in the previous example by QC-flatness of $B^n$, we obtain that
\begin{equation}\label{eq0-ex-qc-finit}
\modn(f(E), f(F), \Om')=\infty.
\end{equation}
On the contrary, let us suppose that $\distM(f(E), f(F))>0$. By the assumption
$\overline{f(E)}$ intersects with $\overline{f(F)}$ only at $\bd \Om'$. Recall, that $\Gamma(f(E), f(F), \Om')$ denotes a family of all rectifiable curves in $\Om'$ joining $f(E)$ and $f(F)$. Therefore, Theorem 7.1 in \cite{va1} leads to the estimate:
\begin{equation}\label{eq-ex-qc-finit}
 \modn(f(E), f(F), \Om')\leq \frac{\mu(\Om')}{\distM(f(E), f(F))}<\infty.
\end{equation}
 This gives us an immediate contradiction with \eqref{eq0-ex-qc-finit}.



  In order to prove the opposite implication, let us choose $E,F \subset \Om'$ such that
$\distM(E, F)=0$. Moreover, let $x\in \overline{E}\cap \overline{F}\cap \bd \Om'$. Finite connectedness of $\Om'$ at the boundary implies that for any $r>0$ and any neighborhood $U\subset B(x,r)$ of $x$ it holds that
  \[
    \overline{E}\cap U\quad \hbox{and}\quad  \overline{F}\cap U
  \]
 belong to the same connectedness component $U(r)$ of $U$, as otherwise $\distM(E, F)>0$. Consider a minimizing sequence of curves $(\ga_i)$ joining $E$ and $F$ in $U(r)$ such that
 \[
  \inf_{i\in \N} \diam\, \ga_i=0\,\,(=\distM(E, F)).
 \]
 Since $B^n$ is locally connected at the boundary and $f$ is a homeomorphism, it holds that
  \[
  \inf_{i\in \N} \diam\, \fv(\ga_i)=0\,=\distM(\fv(E), \fv(F)).
 \]
and the proof of the Property \eqref{class-map-prop} for $f$ is completed.
\end{ex}

Next example shows that class $\calF$ can be larger than the class of quasiconformal mappings.

\begin{ex}\label{ex-notQC}
 Let $B^n$ be a unit ball in $\R^n$ for $n\geq 3$. Define a domain $\Om$ to be a slit ball:
 \[
  \Om:= B^n\setminus \{(x,0,0)\in \R^n: 0\leq x<1\}.
 \]
  There exists no quasiconformal mapping from $\Om$ onto $B^n$ (see Example 17.23(2) in V\"ais\"al\"a~\cite{va1}). Indeed, if such a map existed, then $\Om$ would be Jordan domain by Theorem 17.22 in \cite{va1}, i.e. $\bd \Om$ would be homeomorphic to a sphere $S^{n-1}\subset \R^n$. On the other hand, a non-quasiconformal homeomorphism between $\Om$ and $B^n$ exists and trivially satisfies condition \eqref{class-map-prop} of Definition~\ref{class-map}.

    Similar examples of non-quasiconformal homeomorphisms in class $\calF(B^3, \Om)$ emerge from cusp-type domains $\Om$ studied e.g. by Gehring--V\"ais\"al\"a~\cite{gev} and Gehring~\cite{ge}.
\end{ex}

Next example generalizes observations from Example~\ref{ex1} to the setting of metric spaces. The example appeals to notions of a uniform domain and the Loewner condition. For their definitions and importance we refer to
e.g. Heinonen~\cite[Chapters 8-9]{hei01}, Heinonen--Koskela~\cite{hk} and V\"ais\"al\"a~\cite{vaisala88}. Let us just mention that among examples of uniform domains there are quasidisks, bounded Lipschitz domains and the von Koch snowflake-type domains.

\begin{ex}(Quasi)conformal mappings II) \label{ex-qc2} Let $\Om$ be a uniform domain in a locally compact $Q$-regular $Q$-Loewner space $X$ and $\Om'\subset Y$ be a domain locally connected at the boundary for spaces $X, Y$ as in Definition~\ref{class-map}. Then, quasiconformal mappings as in Definition~\ref{def-qc} belong to class $\calF(\Om, \Om')$. In order to prove this statement we proceed similarly as in Example~\ref{ex1} and distinguish two cases.

With notation of Example~\ref{ex1} if $x\in {\rm Int}\, \Om$, then as above we appeal to the modulus definition of quasiconformal mappings, their injectivity and relations~\eqref{metric-rel} to obtain that condition \eqref{class-map-prop} of Definition~\ref{class-map} holds.

Let now $x\in \bd \Om$. By Fact 2.12 in Herron~\cite{herCam} we know that $\Om$ is QC-flat at $x$, cf. Definition~\ref{qc-flat}. Hence,
\[
\distM(E, F) =0 \,\,\Longrightarrow\,\, \modq(f(E), f(F), \Om')=\infty.
\]
Suppose that $\distM(f(E), f(F))>0$. Thus, $f(E)\cap f(F)=\emptyset$ and we are left with two cases to consider:
\[
\left(\overline{f(E)}\cap \overline{f(F)}\right)\setminus \left(f(E)\cap f(F)\right)=(\not=)\emptyset.
\]
 In both cases the reasoning is similar and, therefore, we present the argument only for the case when there exists $y\in \overline{f(E)}\cap \overline{f(F)}\cap \partial \Om'$. Then, the definition of local connectedness at the boundary gives us that for any ball $B(y, r)$ we find a neighborhood $U\subset B(y, r)$ of $y$ such that $U\cap \Om'$ is connected and path-connected ($Y$ is quasiconvex by Assumptions (A)). Thus, for every $r>0$ we find a curve $\gamma\subset U \cap \Om' \subset B(y, r)$ such that
 \[
  \gamma\cap f(E)\not=\emptyset\not=\gamma\cap f(F).
 \]
 By letting $r\to 0$ we obtain that $\diam \gamma\to 0$ which contradicts assumption that $\distM(f(E), f(F))>0$.

 In order to prove the opposite implication we use the similar approach and assume that $\distM(f(E), f(F))=0$. For any sequence $(\gamma_n)$ of curves joining $f(E)$ with $f(F)$ we consider the sequence of curves $(\fv(\ga_n))$ joining $E$ and $F$. Furthermore, we recall that a uniform domain $\Om$ is locally connected at the boundary (see e.g. \cite[Proposition 11.2]{abbs}). Then, the above reasoning results in $\distM(E, F)=0$.
 \end{ex}

%
%

Another wide class of mappings belonging to $\calF$ is the class of quasisymmetric mappings, see Definition~\ref{def-qsym} and the discussion before it.

\begin{ex}(Quasisymmetric mappings in metric spaces)\label{ex-qs}
 Let $\Om, \Om' \subset X$ be domains locally connected at the boundary. Let further $f:\Om\to f(\Om):=\Om'$ be quasisymmetric (cf. Definition~\ref{def-qsym}). Then $f\in \calF(\Om, \Om')$.

 Suppose that, under the notation of Definition~\ref{class-map}, $\distM(E,F)=0$ for connected subsets $E, F$ in $\Om$. Similarly, to the discussion in above examples it is enough to consider the case when $\overline{E}\cap \overline{F} \cap \bd \Om\not=\emptyset$. Then, there exists a sequence of curves $(\ga_i)$ joining $E$ and $F$ in $\Om$ such that
 \[
 \lim_{i \to \infty} \diam\,\ga_i = 0.
 \]
 Moreover, since $\Om$ is locally connected at the boundary we find a connected set $U\subset \Om$ contained in any ball centered at some $x\in \bd \Om \cap \left(\overline{E}\cap \overline{F}\right)$ and containing all $\ga_i$ for large enough $i$. Proposition 10.8 in \cite{hei01} gives us the following estimate:
 \[
  \diam\,f(\ga_i)\leq \diam\,f(U) \eta \left(2\frac{\diam\, \ga_i}{\diam\,U}\right)\to 0\quad \hbox{for } i\to\infty,
 \]
 where in the last step we again appeal to connectedness of $\Om$ and infer from \cite[Theorem 11.3]{hei01} (see also Tukia--V\"ais\"al\"a~\cite[Corollary 3.12]{tuv}) that $\eta(0)=0$. Since $\eta$ is continuous, we have $\diam\,f(\ga_i)\to 0$ for $i\to\infty$. In a consequence, $\distM(f(E),f(F))=0$. By Proposition 10.6 in \cite{hei01} we know that $\fv$ is quasisymmetric and, thus a reasoning analogous to the one for $f$ gives us also the opposite implication in condition~\eqref{class-map-prop}.
\end{ex}

\section{Main results}\label{sect-mainr}

The main purpose of this section is to show the three types of extension results for mappings in class $\calF$:
\begin{itemize}
\item An extension of a homeomorphism between domains to a continuous map between closures of domains
 (Theorem~\ref{Thm1-homeo-ext}).
\item An extension of a homeomorphism between domains to a homeomorphism between closures of domains (Theorem~\ref{Thm1-homeo-ext-nec}).
\item An extension of a homeomorphism between domains to a homeomorphism between the topological and the prime end boundaries (Theorem~\ref{Thm1-homeo-pext}).
\end{itemize}

We will need the following definitions and results studied in~\cite{abbs}.

First, we recall a notion of accessible boundary points. Such points appear in studies of boundary extension properties of quasiregular and quasiconformal mappings and there are several variants of the following definition
(requiring rectifiability or injectivity of a curve), see e.g. N\"akki~\cite[Section 7.1]{na}.

\begin{definition}[cf. Definition 7.6 in \cite{abbs}]\label{deff-access-pt}
We say that a point $x\in\partial\Om$ is an \emph{accessible} boundary point
if there is a {\rm(}possibly nonrectifiable\/{\rm)} curve $\gamma:[0,1]\to X$ such that $\gamma(1)=x$ and $\gamma([0,1))\subset\Omega$. We call $\ga$ an \emph{end-cut of $\Om$ from $x$}.
%
\end{definition}

%

The following lemma gives us a method of constructing prime ends at accessible boundary points, see Lemma 7.7 in \cite{abbs}. For the sake of completeness of presentation we state this result specializing it to our case. The lemma is used in the proof of the Koebe theorem as well as in the studies of the Royden boundary, see Section~\ref{sect-appl}.

\begin{lem}[Lemma 7.7 in \cite{abbs}]\label{lem-aux}
Let $\Om\subset X$ be a domain satisfying the bounded turning condition. Let $\ga:[0,1]\to X$ be a curve such that $\ga([0,1))\subset\Om$ and $\ga(1)=x\in\bd\Om$.
Let also  $(r_k)$ be a strictly decreasing sequence converging to zero as $k\to\infty$. Then, there exist a sequence $(t_k)$ of positive numbers smaller than $1$ and a prime end $[E_k]$ such that
\begin{enumerate}
\item $I[E_k]=\{x\}$,
\item $\ga([t_k,1))\subset E_k$ and
\item $E_k$ is a component of $\Om\cap B(x,r_k)$ for all $k=1,2,\ldots$.
\end{enumerate}
\end{lem}

If $[E_k]$ is an end and there exists a curve $\ga$ as in the above lemma, then we say that $x\in \bd \Om$ is \emph{accessible through $[E_k]$}.

\begin{rem} \label{rmk-locconn}
Let $X$ be locally path-connected, i.e. every neighborhood of a point $x \in X$ contains a path-connected neighborhood. The Mazurkiewicz--Moore--Menger theorem asserts that if $X$ is a locally connected proper metric space, then $X$ is locally path-connected, see Kuratowski~\cite[Theorem~1, pg~254]{ku}. In particular, every component of an open set is open and path-connected, see \cite[Theorem~2, pg.~253]{ku}.
\end{rem}

Recall the following important observation, see Remark 7.4 in \cite{abbs}.
\begin{rem} \label{rem-connected-diam}
Let $F\subset\Om$ be a a connected set intersecting both sets $A$ and $\Om\setminus A$. Then $F\cap (\Om\cap\bd A)\not=\emptyset$.

An immediate consequence is that if $E_k$, $E_{k+1}$ and $F$ are connected subsets of $\Om$ satisfying
\[
 E_{k+1}\subset E_k, \quad E_{k+1}\cap F \neq \emptyset \quad \hbox{ and }
F\setminus E_k \neq \emptyset,
\]
then $F$ meets both $\Om\cap\bd E_{k+1}$ and $\Om\cap\bd E_k$, which implies in turn that
$\dist(\Om\cap\bd E_{k+1}, \Om\cap\bd E_{k}) \le \diam F$.
\end{rem}

\begin{proof}[Proof of Lemma~\ref{lem-aux}]
 For the proof we refer to Lemma 7.7 in \cite{abbs} and note that $[E_k]$ is an end regardless
 whether in Definition~\ref{def-chain}(\ref{pos-dist}) we consider $d$ or $\distMa$. However, if in Definition~\ref{def-chain}(\ref{pos-dist}) one assumes condition with $\distMa$, then in order to show that $[E_k]$ is a prime end one needs, additionally, to check that also Proposition 7.1, Lemma 7.3 and Remark 7.4 in \cite{abbs} remain true for the Mazurkiewicz distance instead of $d$. This easily follows from Remark~\ref{rmk-locconn} and the bounded turning condition holding for $\Om$.
\end{proof}

We will also need the following description of prime ends for domains finitely connected at the boundary.

\begin{theorem}[Theorem 10.8 in \cite{abbs}]\label{thm-fin-con-homeo}
Assume that $\Om$ is finitely connected at the boundary. Then all prime ends have singleton impressions, and every $x\in\bd\Om$ is the impression of a prime end and is accessible.
\end{theorem}

In the special case of a domain $\Om$ locally connected at the boundary, we are able to provide more detailed construction of prime ends.

\begin{cor}[Corollary 10.14 in \cite{abbs}]\label{cor-loc-conn}
If $\Om$ is locally connected at the boundary and $[E_k]$ is a prime end
in $\Om$, then $I[E_k]=\{x\}$ for some $x\in\bd\Om$ and
there exist radii $r_k^x>0$,  such that $B(x, r_k^x)\cap \Om\subset E_k$, $k=1,2,\ldots$.
Furthermore, for each $x\in\bd\Om$, the sets
\[
 G_k=B(x, 1/k)\cap \Om, \quad k=1,2,\ldots,
\]
define the only prime end $[G_k]$ with $x$ in its impression.
\end{cor}

In what follows, we will often appeal to chains/prime ends $[G_k]$ as \emph{canonical chains/prime ends} and denote them $[G_k^x]$ or, for short, $[G_k]$ if a boundary point associated with $[G_k]$ will be clear from the context of discussion.

Denote by $\bdySP \Om$ the set of all singleton prime ends.

\begin{theorem}[Theorem 10.10 in \cite{abbs}]\label{thm-clOmm-cpt-new}
 A domain $\Om$ is finitely connected at the boundary if and only if $\clOmP=\Om \cup \bdySP \Om$ is compact and all prime ends have singleton impressions.
\end{theorem}

The following auxiliary result shows that given a domain $\Om$ and a sequence $\{E_k\}_{k=1}^{\infty}$ of subdomains in $\Om$ satisfying first two conditions of Definition~\ref{def-chain} and a sequence of points in $\Om$ converging to a point in the impression of $\{E_k\}_{k=1}^{\infty}$, we can infer that every $E_k$ contains almost every element of the sequence of points
(cf. Definition 8.1 in \cite{abbs} for the notion of convergence of a sequence of points to an end).

\begin{lem}\label{lem-conv}
 Let $\Om\subset X$ be a domain satisfying the bounded turning condition. Suppose that $\{E_k\}_{k=1}^{\infty}$ is a sequence of bounded subdomains of $\Om$ such that the following conditions hold:
\begin{enumerate}
\item \label{it-subset-lem}
$E_{k+1}\subset E_k$ for all $k=1,2,\ldots$,
\item \label{pos-dist-lem}
$\distMa(\Omega\cap\bd E_{k+1},\Omega\cap \bd E_k )>0$ for all $k=1,2,\ldots$.
\end{enumerate}
Let $(x_n)$ be a sequence of points in $\Om$ such that $x_n\to x\in I[E_k]$ in $d_X$. Then, for all $k$ there exists $n_k$ such that $x_n\in E_k$ for all $n>n_k$. The same assertion holds if we assume $\dist$ or $\distM$ instead of $\distMa$.
\end{lem}
\begin{proof}
 Suppose that the assertion of the lemma is false. Then, we may find a subsequence of $(x_n)$, denoted $(x_n')\subset \Om$, and $k_{0}\in \N$ such that $x_n'\to x\in I[E_k]$ in $d_X$ and $x_n'\not\in E_{k_0}$ for all $n>n_{k_{0}}$. On the other hand the convergence of $(x_n')$ to $x$ implies that there exists a subsequence $(x_{n_l}')$ such that $x_{n_l}'\to x$ in metric $d_X$ for $l\to\infty$ with a property that $x_{n_l}'\in E_{k_0+2}$ for sufficiently large $l$. Indeed, the bounded turning condition implies that $\Om$ is locally connected at the boundary, see Remark~\ref{rem-lin-con}(3). Hence, Theorem~\ref{thm-clOmm-cpt-new} gives us that all prime ends in $\bdyP \Om$ are singletons. Therefore, $[E_k]$ is equivalent to a singleton prime end given by a chain of shrinking balls centered at $I[E_k]=\{x\}$ intersected with $\Om$.

 Let now $\gamma$ be any continuous path in $\Om$ joining two given $x_{n_l}'$ and $x_n'$. Remark~\ref{rem-connected-diam} and path-connectedness of $\Om$ result in the following inequalities:
 \[
  L d(x_n', x_{n_l}')\geq \diam (\gamma)\geq \distMa(\Omega\cap\bd E_{k_0},\Omega\cap \bd E_{k_0+2}):=c>0,
 \]
 where $L$ is a bounded turning constant of $\Om$. The above estimate holds for all sufficiently large $n$ and $l$ leading us to a contradiction with the convergence of $(x_n')$.
\end{proof}


We are in a position to present the key technical lemma of the paper. In the statement of the result we abuse notation and write $f[E_k]$ to denote an end determined by an image under a map $f$ of any chain $\{E_k\}_{k=1}^{\infty}$ belonging to an end $[E_k]$. However, Part 3 of Lemma~\ref{lem-maps} explains that such a notation is justified.

\begin{lem}\label{lem-maps}
  Let $\Om\subset X$ be a domain locally connected at the boundary and $D\subset Y$ be a domain. Let further $f\in\calF(\Om,D)$. Then the following properties hold:
  \begin{itemize}
  \item[(1)] If $\{E_k\}_{k=1}^{\infty}$ is a chain in $\Om$, then $\{f(E_k)\}_{k=1}^{\infty}$ is a chain in $D$.
  \item[(2)] If $\{F_k\}_{k=1}^{\infty}$ is a chain in $D$, then $\{f^{-1}(F_k)\}_{k=1}^{\infty}$ is a chain in $\Om$. Moreover, if $D$ is additionally finitely connected at the boundary and $[F_k]$ is a (singleton) prime end in $D$, then $[f^{-1}(F_k)]$ is a singleton prime end in $\Om$ as well.

  \item[(3)] If $\{E_k\}_{k=1}^{\infty}$ and $\{F_k\}_{k=1}^{\infty}$ are equivalent chains in $\Om$, then $\{f(E_k)\}_{k=1}^{\infty}$ and $\{f(F_k)\}_{k=1}^{\infty}$ are equivalent chains in $D$, i.e.
      \[
      f[E_k]=[f(E_k)].
      \]
      In particular, if $[E_k]$ is a prime end, then $[f(E_k)]$ is a singleton prime end.
  \end{itemize}
\end{lem}

\begin{proof}
 \underline{Part (1)} Let $\{E_k\}_{k=1}^{\infty}$ be a chain in $\Om$. By Topology it holds that for all $k=1,2,\ldots$ a homeomorphic image $f(E_k)$ of an acceptable set $E_k$ is acceptable (cf. discussion in Part (2) below). We show that $\{f(E_k)\}_{k=1}^{\infty}$ satisfies conditions \eqref{it-subset}-\eqref{impr} of Definition~\ref{def-chain}. In order to show that Part 1 of Definition~\ref{def-chain} holds we note that if $E_{k+1}\subset E_k$, then $f(E_{k+1})\subset f(E_k)$. Indeed, otherwise let $y\in f(E_{k+1})\setminus f(E_k)$ and $x=f^{-1}(y)\in E_k$, since both $f$ and $\fv$ are onto and injective. Then $f(x)=y\in f(E_k)$, contradicting assumption that $f(E_{k+1})\setminus f(E_k)\neq \emptyset$.

 Next we show that $\{f(E_k)\}_{k=1}^{\infty}$ satisfies Part 3 of Definition~\ref{def-chain}. Suppose on the contrary that
 \[
  \bigcap_{k=1}^\infty \overline {f({E}_k)}\nsubseteq \bd D,
 \]
   i.e. there exists $y\in D\cap \bigcap_{k=1}^\infty \overline {f({E}_k)}$. Since, in particular, $y\in D$ then $x:=\fv(y)\in \Om$ as homeomorphisms maps interiors of domains onto interiors. Thus $\dist(x, \bd \Om)>0$. Moreover, there exists $k_0$ such that $x\not \in E_{k_0}$. Otherwise, $x\in E_k$ for all $k=1,\ldots$ and since $\{E_k\}_{k=1}^{\infty}$ is a chain, then by Property \ref{impr} of chains
   \[
   x\in \bigcap_{k=1}^\infty \overline {{E}_k}\subset \bd \Om.
   \]
    Hence, $f(x)=y\not \in f(E_{k_0+1})$ by the previous part of the proof. It follows that $y \not \in f(E_{k})$ for $k\geq k_0+1$ and, hence, also
    \[
     y\not \in \bigcap_{k=1}^\infty f(E_{k}).
    \]
   This leads us to the contradiction.

 In order to show that Property \ref{pos-dist} of Definition~\ref{def-chain} holds for sets $f(E_k)$ we again proceed with the proof by contradiction. Suppose that for some $m$ it holds that
 \begin{equation}\label{eq-main-lem}
 \distMa(D\cap\bd f(E_{m+1}),D\cap \bd f(E_m) )=0.
 \end{equation}
 Then, we find a sequence of curves $(\gamma_k)\subset D$
 \[
  \gamma_k:[0, 1]\to D,\quad \gamma_k(0)\in D\cap \bd f(E_m) \quad\hbox{and}\quad \gamma_k(1)\in D\cap \bd f(E_{m+1}),
 \]
 such that
 \[
  \lim_{k\to\infty}\diam (\gamma_k)=0.
 \]
  The existence of such a sequence of curves follows from the fact since by Assumptions (A) space $Y$ is quasiconvex (cf. Definition~\ref{class-map}), then the minimizing sequence of connected sets, denoted $\{C_k\}_{k=1}^{\infty}$, giving equality in \eqref{eq-main-lem} contains a sequence of curves $\{\ga_k\}_{k=1}^{\infty}$, i.e. $\ga_k\subseteq C_k$ for every $k$.

 Moreover, it holds that $\{\fv(\gamma_k)\}_{k=1}^{\infty}$ is a sequence of curves joining $\Om\cap \bd E_m$ with $\Om \cap \bd E_{m+1}$, and the argument for this observation is similar as in the first part of the proof. By Property \ref{pos-dist} of chain $\{E_k\}_{k=1}^{\infty}$ we have that $\distMa(\Om\cap \bd E_m, \Om \cap \bd E_{m+1})>0$. Since $\Om\cap \bd E_m$ and $\Om\cap \bd E_{m+1}$ are bounded and closed in $\overline{\Om}$ we can choose two subsequences:
  \[
  (\fv(\gamma_{k_l}(0)))\subset \Om\cap \bd E_m\quad\hbox{ and }\quad(\fv(\gamma_{k_l}(1)))\subset \Om\cap \bd E_{m+1}
  \]
  converging to distinct points $x\in \bd \Om$ and $x'\in \bd \Om$, respectively. However, this contradicts Property \eqref{class-map-prop} of mappings in $\calF(\Om, D)$, since then
 \[
  \distM(\fv(\gamma_{k_l}(0)), \fv(\gamma_{k_l}(1)) )>0 \quad \hbox{ for all } l=1,\ldots,
 \]
  which, by Property \eqref{class-map-prop} then implies
 \[
  \distM(f(\fv(\gamma_{k_l}(0))), f(\fv(\gamma_{k_l}(1))) )=\distM(\gamma_{k_l}(0), \gamma_{k_l}(1) )>0 \quad \hbox{ for all } l=1,\ldots.
 \]
 This observation contradicts our assumption that $\diam\,\gamma_{k_l}\to 0$ as $l\to\infty$.

 \underline{Part (2)} Let $\{F_k\}_{k=1}^{\infty}$ be a chain in $D$. Define the following sequence of subsets in $\Om$:
 \[
  E_k:=\fv(F_k),\quad \hbox{ for } k=1, 2, \ldots.
 \]
 It is a direct consequence of $f$ being a homeomorphism, that $\{E_k\}_{n=1}^{\infty}$ are acceptable sets in $\Om$. Indeed, by the continuity of $f$ sets $E_k$ are connected and bounded subset of $\Om$. Suppose now, that
 \[
  \overline{E_k}\cap \bd \Om=\emptyset\quad \hbox{for some }k.
 \]
   Then, we can separate $\overline{E_k}$ and $\bd \Om$ by a compact set $K$. (Note that $\bd \Om$ is compact since it is closed and $X$ is a complete proper space, thus $\bd \Om$ is totally bounded.) Then
   \[
   f(\overline{E_k})\subset f(K)\quad\hbox{and}\quad f(K) \cap \bd D=\emptyset.
   \]
    However, $f(\overline{E_k})=\overline{f(E_k)}=\overline{F_k}$ while by the definition of acceptable sets $\overline{F_k}\cap \bd D\not=\emptyset$ which leads to a contradiction.

 Properties \ref{it-subset} and \ref{impr} of Definition~\ref{def-chain} for a chain $\{E_k\}_{n=1}^{\infty}$ are proven in the same way as the analogous properties showed in the proof of Part 1. It remains to show Property \ref{pos-dist} of chains. Suppose that for some $m$
 \[
  \distMa(\Om\cap \bd E_m, \Om \cap \bd E_{m+1})=0.
 \]
  Property~\eqref{class-map-prop} of mappings in $\calF(\Om, D)$, the definition of chain $\{E_k\}_{n=1}^{\infty}$ together with inequalities~\eqref{metric-rel} and the argument similar to the proof of Property 2 in Part (1) above imply that
 \[
  \distMa(D\cap \bd F_m, D \cap \bd F_{m+1})=\distMa(D\cap \bd f(E_m), D \cap \bd f(E_{m+1}))=0.
 \]

 This contradicts the fact that $\{F_k\}_{k=1}^{\infty}$ is a chain.

 Let us now additionally assume that $D$ is finitely connected at the boundary. Since $\Om$ is locally connected at the boundary, then Theorem~\ref{thm-fin-con-homeo} and Corollary~\ref{cor-loc-conn} imply that a chain
 $\{\fv(F_n)\}_{n=1}^{\infty}$ is divisible by a canonical chain $\{G_k\}_{k=1}^{\infty}$ and, thus
 \[
 I[\fv(F_n)]=\{x\}\subset \bd \Om
 \]
 for some $x\in \bd \Om$.

 \underline{Part (3)} Let chains $\{E_k\}_{k=1}^{\infty}$ and $\{F_l\}_{l=1}^{\infty}$ be equivalent. By the definition this means that these two chains divide each other. Since $\{E_k\}_{k=1}^{\infty}$ divides $\{F_l\}_{l=1}^{\infty}$, then for every $l$ there is $k_0$ such that $E_k\subset F_l$ for all $k\geq k_0$. This immediately implies that $f(E_k)\subset f(F_l)$ for all $k\geq k_0$, as $f$ is a homeomorphism, and so chain $\{f(E_k)\}_{k=1}^{\infty}$ divides $\{f(F_l)\}_{l=1}^{\infty}$. Similarly we prove the opposite property, namely that the division of $\{F_l\}_{l=1}^{\infty}$ by $\{E_k\}_{k=1}^{\infty}$ implies division of $\{f(F_l)\}_{l=1}^{\infty}$ by $\{f(E_k)\}_{k=1}^{\infty}$.

 This observation together with Parts 1 and 2 of the lemma justify the notation $f[E_k]:=[f(E_k)]$ for an end (an equivalence class) represented by a chain $\{f(E_k)\}_{k=1}^{\infty}$.

 It remains to show that impression of $[f(E_k)]$ consists of a single point only. Recall that by Theorem~\ref{thm-fin-con-homeo} we know that $I[E_k]=\{x\} \in \bd \Om$. Suppose, on the contrary, that
 \[
  I[f(E_k)]\supset \{y, z\}\quad \hbox{ for some }\quad y\not=z\in \bd D.
 \]
  Let, further, $(y_k)$ and $(z_l)$ be sequences of points in $D$ converging to $y$ and $z$ respectively. Clearly, $d(y_{k_n}, z_{l_n})>0$ and $y_{k_n}, z_{l_n}\in f(E_m)$ starting from some $m\geq m_0$ and for sufficiently large $n\geq n_0$. Using the path-connectedness of $\Om$ (implied by its rectifiable connectedness) we join every $y_{k_n}$ and $y_{k_{n+1}}$ by a curve, denoted $\ga_n$ and similarly for points $z_{l_n}$ and $z_{l_{n+1}}$, obtaining for every $n$ a curve $\ga_n'$. In a consequence we get a sequence of connected sets $|\ga_n|$ and $|\ga_n'|$ such that
 \[
  \distM(|\ga_n|, |\ga_n'|)>0\quad \hbox{ for } n\geq n_0.
 \]
 However, by setting $\fv(\ga_n)$ and $\fv(\ga_n')$ for $n=1,\ldots$ we arrive at sequences of connected subsets of $E_m$ with the property that for every $m\geq m_0$ we find $n_m$ such that
 \[
 \fv(\ga_n)\subset E_m\quad\hbox{and}\quad\fv(\ga_n') \subset E_m\quad\hbox{for all}\quad n\geq n_m.
 \]
  Moreover, since $I[E_k]=\{x\}$ it holds that
 \[
  \lim_{n\to \infty} \distM(\fv(\ga_n), \fv(\ga_n'))=0,
 \]
 contradicting Property~\eqref{class-map-prop} of class $\calF(\Om, D)$.

 This completes the proof of Part 3 and, thus, the whole proof  of Lemma~\ref{lem-maps} is completed as well.
\end{proof}

\begin{rem}
 Note that in order to prove Parts (1) and (2) of the lemma we only need implication that $\dM(f(E), f(F))=0$ provided that $\dM(E,F)=0$ in condition \eqref{class-map-prop} of Definition~\ref{class-map}. The opposite implication is used only in the proof of Part (3).
\end{rem}

The remaining part of this section is devoted to presentation of various extension results for mappings in class $\calF$ and their consequences. We begin with an observation that for domains with regular enough boundaries mappings in class $\calF$ induce a bijection between the singleton prime ends part of the prime end boundary and the topological boundary of the underlying domains.

\begin{observ}\label{Thm1-cont-ext}
 Let $\Om$ be a domain locally connected at the boundary, $D$ a domain finitely connected at the boundary and $f\in \calF(D, \Om)$. Then, there exists a bijective map $\Phi:\bdySP D \to \bd \Om$ defined as follows: if $[E_k]$ is a (singleton) prime end in $D$, then we assume
 \[
  \Phi([E_k]):=x\in\bd \Om,\quad \hbox{ where }\quad \{x\}=I[\fv(E_n)].
 \]
\end{observ}
\begin{proof}
 Let $[E_k]$ be a prime end in $D$. Then
Theorem~\ref{thm-fin-con-homeo} implies that $[E_k]$ is a singleton
prime end. By Part 2 of Lemma~\ref{lem-maps} we have that
$[\fv(E_n)]$ is a singleton prime end as well. The map $\Phi$ is
well-defined, since equivalent chains belonging to an end $[E_k]$
have the same impressions. Let now $[E_k]$ and $[F_n]$ be two
singleton prime ends in $D$ such that $[E_k]\not=[F_n]$. Part 2 of
Lemma~\ref{lem-maps} implies that $[\fv(E_k)]$ and $[\fv(F_n)]$ are
non-equivalent singleton prime ends and, since $\Om$ is locally
connected at the boundary, it holds that
 \[
  I[\fv(E_k)]=\{x\}\not=\{x'\}=I[\fv(F_n)].
 \]
  From this, injectivity of $\Phi$ follows immediately.

  Let us choose any $x\in \bd \Om$ and consider $[G_k^x]$ a canonical singleton prime end associated with point $x$ as in Corollary~\ref{cor-loc-conn}. Then Part 1 of Lemma~\ref{lem-maps} applied to $\fv$ gives us that $[\fv(G_k^x)]$ defines an end in $\Om$, whereas Part 3 of Lemma~\ref{lem-maps} allows us o conclude that $[\fv(G_k^x)]$ is in fact a singleton prime end. Hence, $\Phi$ is surjective.
\end{proof}

%
\subsection{Continuous extensions for mappings in $\calF$}

 Recall, that for a domain $\Om\subset X$ and a map $f:\Om\to Y$ one defines \emph{the cluster set of $f$ at $x$} as follows:
 \begin{equation}\label{def-cluster0}
  C(f, x):=\bigcap_{U} \overline{f(U\cap \Om)},
 \end{equation}
 where the intersection ranges over all neighborhoods $U$ of $x$ in $X$.

 Let $f:\Om \to \Om'$ be a homeomorphism and $F$ be its continuous extension to $\overline{\Om}$ provided that it exists.  Denote by
 \[
 \Om'':=\Om'\cup \bigcup_{x\in \bd \Om} C(f, x)\subset \overline{\Om'}
 \]
 and by
 \[
 F(\bd \Om):=\bigcup_{x\in \bd \Om} C(f, x)\subset \bd \Om'.
 \]
 Similarly, by $\bdyPF \Om'$ we denote a set of prime ends with impressions in $F(\bd \Om)$. Clearly $\bdyPF \Om'\subset \bdyP \Om'$. We use this notation in the next result.

  In the theorem below we employ singleton prime ends to provide sufficient and necessary conditions for a homeomorphism in class $\calF$ to have a continuous extension.

\begin{theorem}\label{Thm1-homeo-ext}
 Let $\Om$ be a domain satisfying the bounded turning condition and $f\in \calF(\Om, \Om')$ be a homeomorphism. Suppose that $f$ has a continuous extension mapping $F: \overline{\Om}\to \Om''$, then every prime end in the prime end boundary $\bdyPF \Om'$ is a singleton prime end.

 Furthermore, if every prime end in the prime end boundary $\bdyP \Om'$ is a singleton prime end, then every homeomorphism $f\in \calF(\Om, \Om')$ has a continuous extension $F:\overline{\Om}\to \Om''$.
\end{theorem}

The following remark is used in the proof of Theorem~\ref{Thm1-homeo-ext}.

\begin{rem}\label{rem-Thm1-homeo-ext}
(1) Proposition 7.1 in \cite{abbs} asserts that if an end has a
singleton impression, then it is a prime end. If a domain satisfies
the bounded turning condition, then that proposition holds also for
prime ends as in Definition~\ref{def-chain} (with condition
\eqref{pos-dist} stated for the Mazurkiewicz distance $\distMa$
instead of $d$).

\noindent (2) Recall that Proposition 7.2 in \cite{abbs} stays that
$[E_k]$ is a singleton end if and only if $\diam\,E_k\to0$ as
$k\to\infty$.
\end{rem}

\begin{proof}
 Let $F$ be a continuous extension of $f$ as in the statement of the theorem. Suppose that there exists a prime end $[E_k]\in \bdyPF \Om'$ such
that its impression contains more than one point, for instance let
\[
 I[E_k]\supset \{x, x'\}
\]
 for some distinct $x, x'\in F(\bd \Om)$. Since impressions with
more than one point are necessarily continua, we can choose $x$ and
$x'$ arbitrarily close in metric $d_Y$, yet such that the distance
between preimages of $x$ and $x'$ is positive:
 \begin{equation}\label{Thm1-dist}
 d_X(F^{-1}(x), F^{-1}(x'))>0.
 \end{equation}
 Indeed, observe that Part 2 of Lemma~\ref{lem-maps} ensures that
 \[
  F_k:=\fv(E_k)\quad\hbox{for }k=1, 2, \ldots
 \]
 define an end $[F_k]$ in $\Om$. If for all $x\not=x'\in I[E_k]$ it holds that $d_X(F^{-1}(x), F^{-1}(x'))=0$, then $[F_k]$ is a singleton prime end and $[f(F_k)]$ divides $[E_k]$, implying that $[E_k]$ must be a singleton prime end (as $[f(F_k)]$ is such by Part 3 of Lemma~\ref{lem-maps}). This, however, contradicts the assumption that $I[E_k]$ consists of more than one point. Thus, \eqref{Thm1-dist} holds true.

 Moreover,
  \[
  I[F_k]\supset \{F^{-1}(x), F^{-1}(x')\}
  \]
  and under assumption that $x\not=x'$ it holds that $[F_k]$ is not a singleton end. Indeed, if $I[F_k]$ is a singleton end, then $[F_k]$ is a prime end,
since $\Om$ is locally connected at the boundary and
Corollary~\ref{cor-loc-conn} applies. Then
\[
 F^{-1}(x)=F^{-1}(x')
\]
contradicting the assumption that $F$ is an extension of $f$. On the
other hand, let $y\in I[F_k]$. Since $\Om$ is locally connected,
Corollary~\ref{cor-loc-conn} gives us a prime end $[G_k]$ with the
property that $I[G_k]=\{y\}$. By construction $[G_k]$ divides
$[E_k]$. Note that, by Part 3 of Lemma~\ref{lem-maps},
\[
 [E_k']:=[F(G_k)]
\]
 is a prime end and divides $[E_k]$. Since $[E_k]$ is a prime end,
then every end $[E_k']$ dividing $[E_k]$ is equivalent to $[E_k]$,
i.e. $[E_k]$ divides $[E_k']$. Thus, in particular, we get that
\[
 [\fv(E_k)]=[F_k]\quad\hbox{divides}\quad[\fv(E_k')]=[G_k].
\]
 Therefore, $[F_k]$ is a singleton end and from Remark~\ref{rem-Thm1-homeo-ext} we infer that $[F_k]$ is a singleton prime end in $\Om$. We reach a contradiction with the assumptions that $x\not=x'$. Hence $[E_k]$ is a singleton prime end.


 Let us now show the necessity part of the theorem. Suppose now that every prime end in $\bdyP \Om'$ is a singleton prime end. For any $x_0\in \bd \Om$ we can assign a prime end $[E_n]$ in $\bdyP \Om'$ corresponding to $x_0$. The existence of such a prime end follows directly from Corollary~\ref{cor-loc-conn} as a domain satisfying the bounded turning condition is locally connected at the boundary (cf. Example~\ref{ex-domains}). Indeed, let $[G_k]$ be a canonical prime end with a singleton impression $I[G_k]=\{x_0\}$. Part 1 of Lemma~\ref{lem-maps} gives us that $[E_n]=[f(G_k)]$ is an end in $\bdyP \Om'$. By the reasoning similar to the final part of the argument of the sufficiency part of the theorem we obtain that $[E_n]$ is in fact a prime end. Since, upon assuming the opposite, namely that there exists prime end $[E_n']$ dividing $[E_n]$ and not equivalent to $[E_n]$, we obtain that $[\fv(E_n)]$ divides $[G_k]$ and is not equivalent to it. However, this is impossible as $[G_k]$ is a singleton prime end. Moreover, Part 3 of Lemma~\ref{lem-maps} implies that $[E_n]$ is a singleton prime end. Denote $I[E_n]=\{y_{x_0}\}\subset \bd \Om'$. We define $F: \overline{\Om}\to \overline{\Om'}$ by the formula:
 \begin{equation}\label{defn-ext}
 F(x)=
 \begin{cases}
  &f(x),\quad x\in \Om,\\
  &y_x,\quad x\in \bd \Om.
 \end{cases}
 \end{equation}
 From the definition of $F$ it follows that the above defined $F(\bd \Om)= F|_{\bd \Om}\subset \bd \Om'$. In order to show that $F$ is continuous in $\ovOm$ we need to consider three cases.

 \emph{Case 1.} Let $x\in \Om$ and $x_n\to x$ in $d_X$ for a sequence of points $(x_n)\subset \Om$. Then $F\equiv f$ and the continuity of $F$ follows from the continuity of $f$.

 \emph{Case 2.} Let $x\in \bd \Om$ and $x_n\to x$ in $d_X$ for a sequence of points such that $(x_n)\subset \Om$ for large enough $n$. Since $d_X(x_n, x)\to 0$ for $n\to \infty$ we have by Lemma~\ref{lem-conv} that $x_n\in G_k$ for every $k$ and all sufficiently large $n$; hence $f(x_n)\in f(G_k)$. Furthermore, by the above discussion and Remark~\ref{rem-Thm1-homeo-ext} we have that
 \[
 \lim_{k\to \infty}\diam f(G_k)=0\quad\hbox{ and }\quad I[f(G_k)]=\{y_x\}.
 \]
 Therefore, $f(x_n)\to y_x=F(x)$ in $d_Y$ for $n\to \infty$ and the continuity of $F$ follows.

 \emph{Case 3.} Finally, let $x\in \bd \Om$ and $x_n\to x$ in $d_X$ for a sequence of points such that $(x_n)\subset \bd \Om$ for sufficiently large $n$. By the continuity result proven in Case 2, for every $x_n$ we can find such a point $x_n'\in \Om$ that
 \[
  d_X(x_n, x_n')<\frac{1}{n}\quad \hbox{ and } \quad d_Y(F(x_n), F(x_n'))<\frac{1}{n}.
 \]
 By the choice of sequence $(x_n)$ and by the triangle inequality we obtain that
 \[
  d_X(x_n', x)\leq d_X(x_n', x_n)+ d_X(x_n, x)< \frac{1}{n}+d_X(x_n, x)\to 0,\quad \hbox{ as } n\to \infty,
 \]
 and hence $x_n'\to x$. Using again Case 2, we have that for any $\epsilon>0$, there is $n(\epsilon)$ such that for all $n>n(\epsilon)$ it holds $d_Y(F(x_n'), F(x))< \frac{\epsilon}{2}$. The triangle inequality implies that
 \[
 d_Y(F(x_n), F(x))\leq d_Y(F(x_n'), F(x))+d_Y(F(x_n'), F(x_n))<\frac{1}{n}+ \frac{\epsilon}{2}.
 \]
 From this the continuity of $F$ follows, completing the proof in this case and the proof of Theorem~\ref{Thm1-homeo-ext} as well.
\end{proof}

\begin{cor}\label{cor-homeo-ext}
 Let $\Om$ be a domain satisfying the bounded turning condition and $\Om'$ be a domain finitely connected at the boundary. Then, a homeomorphism $f\in \calF(\Om, \Om')$ extends to a continuous mapping $F: \overline{\Om}\to \Om''$. Furthermore, $\Om''=\overline{\Om'}$.
\end{cor}

\begin{proof}
 Theorem~\ref{thm-fin-con-homeo} stays that finitely connected at the boundary domain $\Om'$ has only singleton prime ends. Then, Theorem~\ref{Thm1-homeo-ext} implies immediately the first part of the assertion.

 In order to show the second assertion, let us suppose that $\bd \Om'\setminus \Om''\not=\emptyset$, i.e. there exists \[
 y\in \partial \Om'\setminus \bigcup_{x\in \bd \Om} C(f, x).
 \]
  Theorem~\ref{thm-fin-con-homeo} implies that there is $[E_n]\in \bdySP \Om'$ such that $I[E_n]=\{y\}$. Furthermore, Part 2 of Lemma~\ref{lem-maps} gives us that $[\fv(E_n)]$ is a singleton prime end in $\Om$ with impression denoted by $\{x_y\}$. Then,
 \[
  \fv(E_n)\subset U_n\cap \Om,\quad \hbox{for }\quad n=1,2,\ldots
 \]
  for some neighborhoods $U_n$ of $x_y$. In a consequence, the definition of cluster sets~\eqref{def-cluster0} implies that $y\in C(f, x_y)$ and, hence, $y\in \Om''$ contradicting our initial assumption.

 This completes the proof of the second assertion and the whole proof of the corollary as well.
\end{proof}

\begin{ex}
  Observe that a slit ball and cusp-type domains $\Om$ in Example~\ref{ex-notQC} are locally connected at the boundary and, therefore, homeomorphisms in $\calF(B^3,\Om)$ have continuous extensions to mappings between $\overline{B^3}$ and $\overline{\Om}$.
\end{ex}

\begin{rem}\label{rem-herkos}
 In the setting of quasiconformal analysis in $\R^n$ the results
corresponding to Corollary~\ref{cor-homeo-ext} are due to e.g.
Herron--Koskela, see Theorem 3.3 and Corollary 3.5 in \cite{hek}. In
particular, part (a) of Corollary 3.5 asserts that continuous
extension of a quasiconformal mapping between domains $\Om$ and
$\Om'$ to the topological closure of domains exists if and only if
$\Om$ is QC-flat on the boundary and $\Om'$ is a QED domain, cf.
Definition~\ref{qc-flat} above and \cite[Section 2.B]{hek}. Note
that domains satisfying the bounded turning condition are locally
connected at the boundary, see Part 3 of Remark~\ref{rem-lin-con},
while QED domains are finitely connected at the boundary, cf.
Example~\ref{ex-domains}. Furthermore, Lemma 5.6 in
Martio--Ryazanov--Srebro--Yakubov~\cite{mrsy2} shows that a QC-flat
domain is strongly accessible, which in turn results in finite
connectedness at the boundary, by the argument verbatim to the proof
of Theorem 6.4 in N\"akki~\cite{na3}.

 In the setting of metric spaces an approach to extension problem for quasiconformal mappings is presented in Andrei~\cite{andr}.
\end{rem}

%
\subsection{Homeomorphic extensions for mappings in $\calF$}

 In our next results we describe conditions for existence of homeomorphic extensions of mappings in $\calF$. In the setting of quasiconformal mappings in $\R^n$ the corresponding results can be found in V\"ais\"al\"a~\cite{va1}, cf. Theorems 17.17, 17.18 and 17.20, also in Herron--Koskela~\cite{hek}, Theorem 3.3(e) and Corollaries 3.5(c) and 3.6. Conditions presented in \cite{va1} involve collared domains, see Definition~\ref{def-collar} and Jordan domains, while conditions in \cite{hek} appeal to quasiextremal distance domains, QC-flat and QC-accessible domains, see Sections 2.B and 3.A in \cite{hek}. Below we take different approach and apply prime ends to ensure existence of homeomorphic extensions.

\begin{theorem}\label{Thm1-homeo-ext-nec}
 Let $\Om$ be a domain satisfying the bounded turning condition and $D$ be a domain. Then, a homeomorphism $f\in \calF(\Om, D)$ extends to a homeomorphism $F: \overline{\Om}\to \overline{D}$ if and only if:
 \smallskip 
  
\noindent (1) non-equivalent prime ends in $\bdyP D$ have distinct impressions, and\\
 (2) every prime end in $\bdyP D$ is a singleton prime end.
\end{theorem}

Note that by Theorem~\ref{thm-fin-con-homeo} assumptions of the second part of Theorem~\ref{Thm1-homeo-ext-nec} are satisfied by e.g. domains locally connected at the boundary.

\begin{rem}
 In general a map in class $\calF(\Om, D)$ need not preserve the local connectivity at the boundary. Indeed, if $\Om$ is a unit disc in $\R^2$ and $D\subset \R^2$ is a slit-disc, then direct computations show that a Riemann map $f\in \calF(\Om, D)$ even though $D$ fails to be locally connected at the boundary. Nevertheless, if a map $f\in \calF(\Om, D)$ extends to a homeomorphism $F:\overline{\Om}\to \overline{D}$, then local(finite) connectivity at the boundary of $\Om$ is preserved under $F$ and so $D$ is also locally(finite) connected at the boundary, cf. Remark 17.8 in \cite{va1}. Therefore, the sufficiency part of Theorem~\ref{Thm1-homeo-ext-nec} additionally implies that $D$ is locally connected at the boundary.
\end{rem}

\begin{proof}[Proof of Theorem~\ref{Thm1-homeo-ext-nec}]
 Let $F$ be a homeomorphic extension of $f$ from $\ovOm$ onto $\ovD$. By Part 3 of Remark~\ref{rem-lin-con} and  Corollary~\ref{cor-loc-conn} prime end boundary $\bdyP \Om$ consists of singelton prime ends only. Then, Parts 1 and 3 of Lemma~\ref{lem-maps} imply that $[f(E_n)]\in \bdySP D$ for every prime end $[E_n]\in \bdySP \Om=\bdyP \Om$. Suppose that there exists a prime end
 \[
  [F_n]\in \bdyP D\setminus \bdySP D.
 \]
 Then, a set
 \[
  Y:=F^{-1}(I[F_n])\subset \bd \Om
 \]
  and for all $y\in Y$ it holds that $[F(G_k^y)]\in \bdySP D$, where $[G_k^y]$ is a canonical prime end in $\Om$ as constructed in Corollary~\ref{cor-loc-conn}. However, then we obtain that
  \[
   F(Y)\cap I[F_n]=\emptyset,
  \]
  which contradicts the definition of $Y$. Thus, $\bdyP D=\bdySP D$.

  Since $F$ and $F^{-1}$ are injective, then it is not possible that two nonequivalent prime ends $[E_k]$ and $[F_n]$ in $D$ have the same (singleton) impressions, i.e. $I[E_k]=\{x\}=I[F_n]$ for some $x\in \bd D$.

 For the opposite implication suppose that $\bdySP D=\bdyP D$ and that non-equivalent prime ends have distinct impressions. Then, Theorem~\ref{Thm1-homeo-ext} implies continuity of extension map $F$ defined in \eqref{defn-ext}.  In order to prove injectivity of $F$ let us consider two cases: first, for distinct points $x, y\in \Om$ or $x\in \Om, y\in \bd \Om$ we clearly have that
 \[
 F(x)=f(x)\not=f(y)=F(y)\quad \hbox{ or, respectively }\quad F(x)=f(x)\not=F(y),
 \]
  due to $f$ being a homeomorphism and the definition of $F$. Suppose now that for distinct $x, y\in \bd \Om$ we have
  $F(x)=F(y)=z\in \bd D$. Local connectivity at the boundary of $\Om$ implies existence of canonical prime ends
  $[G_k^x]$ and $[G_k^y]$ in $\Om$ such that
  \[
   I[G_k^x]=\{x\}\not=\{y\}=I[G_k^y].
  \]
  If $I[f(G_k^x)]=I[f(G_k^y)]=\{z\}$, then by assumption (2) on $\bd D$ we have that every $z\in \bd D$ is an impression of some prime end in $D$, denoted $[H_n]$. Moreover, both prime ends $[f(G_k^x)]$ and $[f(G_k^y)]$ divide $[H_n]$ and, therefore, are equivalent (equal) to $[H_n]$ and to each other (by the definition of prime ends). Then, Part 2 of Lemma~\ref{lem-maps} applied to $\fv$ gives us that also $[G_k^x]=[G_k^y]$. More precisely, we get that $\{\fv(f(G_k^x))\}$ and $\{\fv(f(G_k^y))\}$ are chains (and thus give rise to ends) in $\Om$ and are equivalent dividing both $[G_k^x]$ and $[G_k^y]$. This contradiction completes the proof for injectivity of $F$.

  In order to show that $F$ is onto let $x\in \bd D$. Then, by assumption (2) on $\bd D$, there exists a prime end $[E_n]$ with $I[E_n]=\{x\}$. Part 2 of Lemma~\ref{lem-maps} gives us that $[\fv(E_n)]$ is a singleton prime end in $\Om$. Indeed, if an end $[\fv(E_n)]$ has a non-singleton impression $I$, then it is divisible by a canonical prime end $[G_k^x]$ for some $x\in I$ and $[f(G_k^x)]$ is equivalent to $[E_n]$. By the reasoning as in the beginning of the proof for Part 3 of Lemma~\ref{lem-maps} applied to $\fv$, we obtain that $[\fv(f(G_k^x))]=[G_k^x]$ and $[\fv(E_n)]$ are equivalent. However, this is impossible, since $[G_k^x]$ is a singelton prime end. Thus, $[\fv(E_n)]$ is a singleton end, which by Part (1) of Remark~\ref{rem-Thm1-homeo-ext} implies that it is a singleton prime end. Therefore, by Corollary~\ref{cor-loc-conn} there is a point $y\in \bd \Om$ such that $I[\fv(E_n)]=\{y\}$. Moreover, by construction $F(y)=x$.

  Finally, $F^{-1}$ is continuous by the argument similar to the one for continuity of map $F$ in Theorem~\ref{Thm1-homeo-ext} and, therefore, will be omitted.
 \end{proof}

\subsection{Homeomorphic extensions for mappings in $\calF$ with respect to topological and prime end boundaries}
In the result below we deal with the extension of the Mazurkiewicz distance between sets to the setting of prime ends. Let $\Om\subset X$ be a domain finitely connected at the boundary and let $[E_k]$ and $[F_n]$ be prime ends in $\bdyP \Om$. By Theorem~\ref{thm-fin-con-homeo} we know that $\bdySP \Om=\bdyP \Om$ and so all prime ends are singletons. We set
\begin{equation*}\label{distMst-def}
 \distMst([E_k], [F_n])= \distMa(I[E_k], I[F_n]).
\end{equation*}
Since equivalent chains representing given prime ends have the same impressions, the distance is well defined and defines a metric in $\bdySP \Om$ by the properties of the Mazurkiewicz metric $\distMa$. Similarly we define
\[
 \distMst(x, [E_k])= \distMa(x, I[E_k])\quad \hbox{ for }x\in \Om, [E_k]\in \bdySP
 \Om,
\]
and $\distMst(x,y)= \distMa(x,y)$ for $x,y \in \Om$.


The following theorem is the main result of this section.

\begin{theorem}\label{Thm1-homeo-pext}
 Let $\Om$ be a domain satisfying the bounded turning condition and $D$ be a domain finitely connected at the boundary. Then, a homeomorphism $f\in \calF(\Om, D)$ extends to a homeomorphism $F_{\rm P}: \ovOm \to D\cup \bdyP D$ continuous with respect to the Mazurkiewicz distance $\distMst$. Namely, we define
 \begin{equation}\label{defn-prime-ext}
 F_{\rm P}(x)=
 \begin{cases}
  &f(x),\quad x\in \Om,\\
  &[f(G_k^x)],\quad x\in \bd \Om,
 \end{cases}
 \end{equation}
  where $[G_k^x]$ is a unique prime end in $\bdySP \Om$ given by Corollary~\ref{cor-loc-conn} such that $I[G_k^x]=\{x\}$.
\end{theorem}

 Before proving this result let us provide some of its consequences and compare it to previous extension results.

\begin{cor}[cf.~Theorem 4.2 in N\"akki~\cite{na}]\label{Thm1-quasiconf-pext}
 Let $B^n\subset \R^n$ be a ball and $D\subset \R^n$ be collared. Then a quasiconformal homeomorphism $f:B^n \to D$ extends to a homeomorphism $F_{\rm P}: \overline{B^n} \to D\cup \bdyP D$.
\end{cor}
\begin{proof}
 A ball $B^n\subset \R^n$ satisfies the bounded turning condition and if $D\subset \R^n$ is collared (see Definition~\ref{def-collar}), then it is locally connected at the boundary, see Theorem 17.10 in \cite{va1}. By the discussion in Example~\ref{ex1}(a) above, we get that $f\in \calF(B^n, D)$ and, thus Theorem~\ref{Thm1-homeo-pext} results in the assertion of the lemma.
\end{proof}

 Similarly, by Example~\ref{ex-qc-finit} and Theorem~\ref{Thm1-homeo-pext}, we show Corollary~\ref{Thm1-quasiconf-pext-finit} below. There we retrieve the extension result for quasiconformal mappings discussed in Section 3.1 in V\"ais\"al\"a~\cite{va2}. The prime end boundary studied in \cite{va2} is defined via accessible curves and V\"ais\"al\"a's proof is based on the property that the impression of every singleton prime end in finitely connected at the boundary domains is accessible, cf. Theorem~\ref{thm-fin-con-homeo} and Lemma~\ref{lem-aux} for the relation between prime ends and accessibility.

 In particular, if $f$ is conformal, $n=2$ and $D$ is simply-connected, then we obtain a counterpart of the celebrated Carath\'eodory extension theorem, see \cite{car1}. However, our result differs from Carath\'eodory's theorem as we use a different prime ends theory than in \cite{car1}.

\begin{cor}[cf. Section 3.1 in~\cite{va2}]\label{Thm1-quasiconf-pext-finit}
 Let $B^n\subset \R^n$ be a ball and $D\subset \R^n$ be a domain finitely connected at the boundary. Then a quasiconformal homeomorphism $f:B^n \to D$ extends to a homeomorphism $F_{\rm P}: \overline{B^n} \to D\cup \bdyP D$.
\end{cor}

Let us also add that in the setting of quasiconformal mappings in domains in the Heisenberg group $\mathbb{H}_{1}$ the result corresponding to Theorem~\ref{Thm1-homeo-pext} has been recently proven in \cite[Theorem 3.7]{aw2}. There, we assume that $\Om$ is collared and the theory of prime ends studied in \cite{aw2} differs from the one above.

Finally, motivated by a slit-ball domain, see Example~\ref{ex-notQC}, Theorem 17.22 and Example 17.23 in \cite{va1} we provide an example of non-quasiconformal homeomorphism which satisfies the assertion of Theorem~\ref{Thm1-homeo-pext}.
\begin{cor}
  Let $B^n\subset \R^n$ be a ball and $D\subset \R^n$ be a a slit-ball, for $n>2$. Then a homeomorphism $f \in \calF(B^n, D)$ extends to a homeomorphism $F_{\rm P}: \overline{B^n} \to D\cup \bdyP D$.
\end{cor}
\begin{proof}[Proof of Theorem~\ref{Thm1-homeo-pext}]
 By assumptions and Corollary~\ref{cor-homeo-ext} we obtain that a homeomorphism $f\in \calF(\Om, D)$ extends continuously to $F: \ovOm \to \ovD$.

 Lemma~\ref{lem-maps} implies that prime ends in $\bdyP \Om$ correspond to prime ends in $\bdyP D$ under a mapping in $\calF(\Om, D)$ and opposite. This observation allows us to define
  \[
   F_{\rm P}: \ovOm \to D\cup \bdyP D,
  \]
   a unique bijective extension of $f$. Namely, let $x\in \bd \Om$ and $[G_k^x]$ be a unique prime end in $\bdySP \Om$ given by Corollary~\ref{cor-loc-conn} such that $I[G_k^x]=\{x\}$. We define $F_{\rm P}$ as in \eqref{defn-prime-ext} above.
 Mapping $F_{\rm P}$ is well-defined by the definition of prime ends. The injectivity of $F_{\rm P}$ follows from the injectivity of $f$ and the following argument:

 Suppose there exist $x\not=y \in \bd \Om$ such that
 \[
  F_{\rm P}(x)=F_{\rm P}(y)=[E_k]\in \bdyP D.
 \]
  Then by Lemma~\ref{lem-maps}(2) we have that $[\fv(E_k)]\in \bdyP \Om$ and divides, and hence is equivalent to, both canonical prime ends $[G_k^x]$ and $[G_k^y]$. In a consequence $[G_k^x]=[G_k^y]$ and $x=y$ by Theorem~\ref{thm-fin-con-homeo} and Corollary~\ref{cor-loc-conn}.

  In order to show that $F_{\rm P}$ is onto let $[E_n]\in \bdySP D$ and note that $\bdySP D=\bdyP D$, as $D$ is finitely connected at the boundary. Hence, by Part 2 of Lemma 3, $[\fv(E_k)]$ is a singleton prime end in $\Om$ and divides a canonical prime end $[G^x_k]$ for some $x\in \bd \Om$. Thus, $[\fv(E_k)]=[G^x_k]$ by the definition of prime ends (the minimality property with respect to division of ends). In particular, $F^{-1}_{\rm P}$ exists and is well-defined.

 Next, we show the continuity of $F_{\rm P}$. Let $[E_n]\in \bdyP D$ and set
 \[
  x:=F^{-1}_{\rm P}([E_n])\, \hbox{ with }\, x\in \bd \Om.
  \]
   Let $\epsilon>0$. Since $F$ is continuous map from $\Om$ to $D$ with respect to metrics $d_X$ and $d_Y$, respectively, then there exist $\delta=\delta(\epsilon)>0$ and a set
 \[
 U:=\ovOm\cap B(x, \delta)\quad\hbox{ such that }\quad F(U)\subset B(I[E_n], \epsilon)\cap D.
 \]
 Then
 \[
  \distMst(F(y), I[E_n])<\epsilon\quad\hbox{  for all }\, y\in U
 \]
  and hence, $F_{\rm P}$ is continuous at $x$ in metric $\distMst$.

 Finally, we show that $F^{-1}_{\rm P}$ is continuous with respect to metrics $\distMst$ in $D$ and $d_X$ in $\Om$. Then, the proof that $F_{\rm P}$ is a homeomorphism will be completed.

 For a set $U$ defined as above we choose a compact set
 \[
 G\subset \ovOm\setminus F^{-1}(I[E_n])
 \]
 separating $x$ and $\bd B(x, \epsilon)\cap \ovOm$ in $\ovOm$. That such a compact set exists follows from the fact that $\bd B(x, \epsilon)\cap \ovOm$ is a closed subset of a compact domain $\ovOm$, and thus compact as well. By taking into account that $\{x\}$ is compact and that the metric spaces $X$ and $Y$ are Hausdorff, the existence of set $G$ is proved. Since $d_X(G, x)>0$, then $\distM(G, x)>0$ by inequalities \eqref{metric-rel}. Next, observe that by the above presentation and the definition of the Mazurkiewicz distance, cf. Preliminaries, we have that
 \[
  \distMa(F(y), I[E_n])=\inf \diam\,K \leq \diam \ga_{F(y), I[E_n]}\leq \epsilon,
 \]
 where $\ga_{F(y), I[E_n]}$ denotes a curve joining $F(y)$ and $I[E_n]$ in $U$ whose existence follows from quasiconvexity of space $Y$, see Assumptions (A) in Section~\ref{sect-map-class} and Remark~\ref{rmk-locconn}. This observation together with \eqref{metric-rel} imply that
 \[
   \distMa(F(y), I[E_n])\approx  d(F(y), I[E_n]).
 \]

 Therefore, $\distM(F(G), I[E_n])>0$ as $F$ is a homeomorphism and $F|_{\Om}=f$ belongs to class $\calF(\Om, D)$. Finally, let us choose
 \[
  \delta:=\distMa(F(G), I[E_n])<\distM(F(G), I[E_n]).
 \]
 Let now $[F_k]\in \bdyP D$ be such that
 \[
 \distMst([E_n], [F_k])=\distMa(I[E_n], I[F_k])\leq \distM(I[E_n], I[F_k])<\delta.
 \]
 Then, $F_{\rm P}([F_k])\in U$ and so, $F^{-1}_{\rm P}$ is continuous in $\distMst$.
\end{proof}
%
%
\section{Applications}\label{sect-appl}

 This section is devoted to some applications of extension results and prime ends. First, we discuss and prove a variant of the Koebe theorem on arcwise limits along end-cuts for mappings in class $\calF$. In the second part, we relate the prime end boundary as defined in Section~\ref{sec-pet} and the Royden boundary. The latter one type of the boundary arises naturally in the extension problems, see Theorem~\ref{thm51-sod}.

\subsection{The Koebe theorem}\label{sect-koebe}

In 1915 Koebe~\cite{ko} proved that a conformal mapping from a
simply-connected planar domain $\Om$ onto the unit disc has arcwise
limits along all end-cuts of $\Om$. The purpose of this section is
to show a counterpart of Koebe's result for mappings in the class
$\calF$ in metric spaces. First, we need to introduce two auxiliary
definitions.


 In the Euclidean setting the following definition appears in N\"akki~\cite{na1, na2} as an uniform domain. However, in order to avoid confusion with uniform domains studied above and by e.g. V\"ais\"al\"a~\cite{vaisala88} (cf. Definition 11.1 in \cite{abbs}) we shall use a term mod-uniform domains instead. See also \cite[Definition 4.1]{aw2} for mod-uniform domains considered in the setting of the Heisenberg group $\mathbb{H}_1$.
\begin{definition}\label{def-uni-dom}
 We say that a domain $\Om \subset X$ is \emph{mod-uniform} if  for every $t>0$ there is a $\epsilon>0$ such that if $$
 \min\{\diam(E), \diam(F)\}\ge t,\quad \hbox{then }\quad \Mod_{Q}(E, F, \Om)\geq \epsilon.
 $$
\end{definition}

In order to illustrate the definition, let us recall that by
Herron~\cite[Fact 2.12]{herCam} a uniform subdomain in a locally
compact $Q$-regular $Q$-Loewner space is mod-uniform.

 Next, we refine the definition of a cluster set and include the behavior of a mapping along an end-cut, cf. \eqref{def-cluster0} and Definition~\ref{deff-access-pt}.

\begin{definition}\label{def-cluster}
 Let $\Om\subset X$ be a domain, $f:\Om\to Y$ be a mapping and $x\in \bd \Om$. We say that a sequence of points $\{x_n\}_{n=1}^{\infty}$ in $\Om$ \emph{converges along an end-cut $\ga$ at $x$} if there exists a sequence $\{t_n\}_{n=1}^{\infty}$ with $0<t_n<1$ such that $\lim_{n\to\infty} t_n=1$ and $x_n=\ga(t_n)$ that
 \[
 \lim_{n\to \infty}d_X(x_n, x)=0.
 \]
 We say that a point $x'\in X$ belongs to \emph{the cluster set of $f$ at $x$ along an end-cut $\ga$ from $x$} denoted by $C_{\ga}(f, x)$, if there exists a sequence of points $\{x_n\}_{n=1}^{\infty}$ converging along $\ga$ at $x$, such that
 \begin{equation*}
 \lim_{n\to\infty} d_Y(f(x_n), x')=0.
 \end{equation*}

 If $C_{\ga}(f, x)=\{y\}$, then $y$ is called an \emph{arcwise limit(asymptotic value)} of $f$ at $x$.
\end{definition}

The following result extends Koebe's theorem and Theorem 7.2 in N\"akki~\cite{na} to the setting of mappings in $\calF$. Moreover, we study more general end-cuts than in \cite{na}.

\begin{theorem}[The Koebe theorem in metric spaces]\label{thm-koebe}
 Let $f\in \calF(\Om, D)$ be a homeomorphism between a domain $\Om\subset X$ satisfying the bounded turning condition and a domain $D \subset Y$ finitely connected at the boundary. Then $f$ has arcwise limits along all end-cuts of $\Om$.
\end{theorem}

\begin{rem}
 A variant of the Koebe theorem for quasiconformal mappings between domains in the Heisenberg group $\mathbb{H}_1$
 such that one domain is finitely connected at the boundary and the target domain is mod-uniform is proved in \cite{aw2}. As observed by N\"akki, mod-uniform domains in $\overline{R}^n$ are finitely connected at the boundary, see Theorem 6.4 in \cite{na3}.  However, a domain $\Om\subset \overline{R}^n$ finitely connected at the boundary is mod-uniform if and only if $\Om$ can be mapped quasiconformally onto a collared domain, see \cite[Section 6.5]{na3}.
\end{rem}

\begin{proof}[Proof of Theorem~\ref{thm-koebe}]
  Since by Remark~\ref{rem-lin-con}(3) we have that $\Om$ is locally connected at the boundary, then Theorem~\ref{thm-fin-con-homeo} implies that every point in the boundary $\bd \Om$ is accessible. Let $x\in \bd \Om$ and suppose that $\ga$ is an end-cut in $\Om$ from $x\in \Om$ (cf. Definition~\ref{deff-access-pt}). By Lemma~\ref{lem-aux} there exists a singleton prime end $[E_k]$ such that $x$ is accessible through $[E_k]$ (cf. the statement following Lemma~\ref{lem-aux}). Let
  \[
   x_n:=\ga(t_n) \quad\hbox{ for } n=1,2,\ldots
  \]
  and some $t_n\to 1$ as $n\to \infty$ be any sequence of points converging to $x$ in $d_X$. Since every $x_n$ belongs to some acceptable set $E_{k(n)}$ in $[E_n]$ and by Part 3 of Lemma~\ref{lem-maps} it holds that $[f(E_{k(n)})]$ is a singleton prime end in $D$, we obtain a corresponding sequence of points $f(x_n)$ in $D$ for $n=1, 2, \ldots$ which converges to
  \[
  I[f(E_{k(n)})]=x'\in \bd D.
  \]
  Moreover, Theorem~\ref{thm-fin-con-homeo} together with Theorem~\ref{Thm1-homeo-ext} (Corollary~\ref{cor-homeo-ext}) give us that $F$, a continuous extension of $f$ to $\overline{\Om}$, exists and it holds that
  \[
   \lim_{n\to \infty} d_Y(f(x_n), F(x))=0,\quad \hbox{ and } F(x)=x'.
  \]
 Since every prime end in $D$ is a singleton prime end, therefore it
holds that the cluster set of $f$ at $x$ along an end-cut $\ga$ from
$x$ consists of a single point only, namely $C_{\ga}(f, x)=\{x'\}$.
By Definition~\ref{def-cluster}, this implies that $f$ has an
arcwise limit at $x$.
 \end{proof}

 By narrowing the class of mappings to the quasiconformal ones (cf. Definition~\ref{def-qc}) we may weaken assumptions in Theorem~\ref{thm-koebe} on $\Om$ by the price of strengthening slightly requirements on the target domain. In a consequence we obtain the following result.

 \begin{cor}
  Let $X$ be a path-connected doubling metric space. Let $f:\Om\to D$ be a quasiconformal map between a domain $\Om\subset X$ finitely connected at the boundary and a mod-uniform domain $D\subset X$. Then $f$ has arcwise limits along all end-cuts of $\Om$.
 \end{cor}

 The proof closely follows steps of the corresponding proof of the
Koebe theorem for quasiconformal mappings in the Heisenberg group,
see Theorem 5.1 in \cite{aw2}. Hence, we omit the proof.

\subsection{Prime ends and Royden boundaries}\label{sect-royden}

 In this section we relate our results on prime ends with homeomorphic extensions of quasiconformal mappings and the theory of Royden boundaries. Let us briefly set up the stage for our considerations and recall necessary notions from the theory of Royden algebras and compactifications. Our presentation is based on a work by Soderberg~\cite{sod}.

 In what follows let $\Om,\Om'$ be domains in $\R^n$ for $n\geq 2$. We define a \emph{Royden algebra on $\Om$} and denote $\mathcal{A}(\Om)$, as an algebra of all bounded continuous functions $u:\Om\to\R$ with pointwise addition and multiplication such that first order weak partial derivatives of $u$ exist and belong to $L^1(\Om)$. The norm of $u\in \mathcal{A}(\Om)$ is defined as follows:
  \[
  \|u\|_{\Om}:=\|u\|_{L^{\infty}(\Om)}+\|\nabla u\|_{L^{n}(\Om)}.
  \]
  Such an algebra is a commutative regular Banach algebra which separates points in $\Om$ and is inverse-closed (see Preliminaries in \cite{sod} for further details).

 It turns out that there is a correspondence between quasiconformal mappings and Royden algebras isomorphisms. Namely, a quasiconformal map $f:\Om\to\Om'$ defines an algebra isomorphism $f^*:\mathcal{A}(\Om')\to \mathcal{A}(\Om)$ by the formula
 \[
  f^*v:=v\circ f\quad \hbox{for any }\, v\in \mathcal{A}(\Om').
 \]
  Furthermore, an algebra isomorphism between $\mathcal{A}(\Om')$ and $\mathcal{A}(\Om)$ induces a quasiconformal map $f:\Om\to\Om'$, cf. \cite[Theorem 1.1]{sod}. Similarly to the setting of prime end boundaries (see Theorem~\ref{Thm1-homeo-pext}), quasiconformal mappings give rise also to maps between an ideal type of boundary defined via Royden algebras. Namely, a collection of all non-zero, bounded linear homomorphisms $\chi: \mathcal{A}(\Om)\to \R$ is called a \emph{Royden compactification $\Om^{*}$}. Thus, $\Om^{*}\subset \mathcal{A}(\Om)'$ a dual space of $\mathcal{A}(\Om)$. Moreover, $\Om^{*}$ is a compact Hausdorff space in the relative weak$^*$ topology generated by $\mathcal{A}(\Om)$.
 Points $x\in \Om$ can be identified with a subset of $\Om^*$ denoted $\hat{\Om}$ via point evaluation homomorphisms
 \[
  \hat{x}(u)=u(x)\quad \hbox{ for any }u\in \mathcal{A}(\Om).
 \]
 Hence, the identification $x\to \hat{x}$ defines a homeomorphic embedding of $\Om$ onto the image $\hat{\Om}\subset \Om^{*}$. We define \emph{the Royden ideal boundary} of $\Om$ as follows, cf. \cite{sod}:
 \begin{equation*}
  \Delta=\Delta_{\Om}:=\Om^{*}\setminus \hat{\Om}.
 \end{equation*}
 Let $T=f^*$ be the Royden algebra isomorphism defined above for a given quasiconformal map $f$. One defines an adjoint operator $T^*:\Om^*\to (\Om')^*$ by the formula
  \[
  T^*\chi=\chi\circ T\quad \hbox{ for any }\quad \chi\in \Om^*
  \]
 and show the following result (cf. \cite[Theorem 2.3]{sod}):
\emph{Operator $T^*$ is a homeomorphic extension of $f$ to the
Royden compactification $\Om^*$ with the property that $T^*$ maps
the Royden ideal boundary of $\Om$ onto the corresponding Royden
ideal boundary of $\Om'$.} Related is the following notion crucial
for the further discussion. Let $x\in \overline{\Om}$, then we
denote by $\Phi_x\subset \Om^{*}$ a \emph{fiber over $x$}, i.e.
 \begin{equation*}
  \Phi_x:=
  \begin{cases}
  &\chi=\hat{x},\qquad x\in \Om\\
  &\chi \hbox{ corresponding to a Royden net converging to }x,\qquad x\in \bd \Om.
  \end{cases}
 \end{equation*}
 We refer to Sections 3 and 4 in \cite{sod} for definitions and properties of Royden nets, in particular see \cite[Definition 3.3]{sod}. It turns out that $\Phi_x$ is a compact subset of $\Delta$, also that different boundary points have distinct fibers. Moreover, $\Delta=\bigcup_{x\in \bd \Om}\Phi_x$.

 The following result connects the topic of our work to Royden algebras and the Royden compactification.

 \begin{theorem}[Theorem 5.1 in \cite{sod}]\label{thm51-sod}
   Let $f:\Om\to\Om'$ be a quasiconformal map and $T=f^*:\mathcal{A}(\Om')\to \mathcal{A}(\Om)$ be the corresponding Royden algebras isomorphism. A homeomorphic extension of $f$, denoted $F:\overline{\Om}\to\overline{\Om'}$, exists if and only if for every $x\in \bd \Om$ there is a $y\in \bd \Om'$ such that $T^*(\Phi_x)=\Phi_y$. Moreover, it holds that $F(x)=y$.
 \end{theorem}

 We are in a position to present the main result of this section (cf. Theorem 7.4 in \cite{sod} for a different prime ends theory).

 \begin{theorem}\label{thm-Royden-prime}
  Let $\Om\subset \R^n$ be a domain finitely connected at the boundary. Then, for every point $x\in \bd \Om$ it holds that the set of components of a fiber $\Phi_x$ coincides with the set of prime ends with impressions $\{x\}$.
 \end{theorem}

 In the case of John domains, Theorem~\ref{thm-Royden-prime} allows us to provide the following estimate for a number of components for fibers in the Royden compactification.

 \begin{cor}\label{cor-royden-john}
  Let $\Om\subset \R^n$ be a John domain with the John constant $C_\Om$. Then, for any $x\in \bd \Om$ the number of components of a fiber $\Phi_x$ is at most $N(n, C_\Om)$.
 \end{cor}

 \begin{proof}
  Theorem 11.3 in \cite{abbs} asserts that for a John domain $\Om\subset \R^n$ there exists a constant $N$, depending only on the doubling constant $C_\mu=2^n$, the John constant $C_\Om$ and the quasiconvexity constant, such that $\Om$ is at most $N$-connected at every boundary point. Then, Proposition 10.13 in \cite{abbs} gives us that every point $x\in \bd \Om$ is the impression of exactly $N$ distinct prime ends.
  Moreover, there is no other prime end with $x$ in its impression. These propositions together with Theorem~\ref{thm-Royden-prime} result in the assertion of the corollary.
  \end{proof}

 The proof of Theorem~\ref{thm-Royden-prime} requires the following observation, specialized to the Euclidean setting, about the structure of prime ends for domains finitely connected at the boundary. One can think about the lemma as a counterpart of the construction of canonical prime ends obtained in Corollary~\ref{cor-loc-conn} for the setting of domains locally connected at the boundary. The main difference is that now every boundary point can be the impression of more than one prime end.

 \begin{lem}[Lemmas 10.5 and 10.6 in \cite{abbs}]\label{lem10-5}
 Assume that a domain $\Om\subset \R^n$ is finitely connected at $x_0\in\bd \Om$. Let $A_k\subsetneq\Om$ be such that:
\begin{enumerate}
\item $A_{k+1}\subset A_k$,
\item $x_0\in \overline{A}_k$,
\item $\dist(x_0,\Om\cap\bd A_k)>0$ for each $k=1,2,\ldots$.
\end{enumerate}
Furthermore, let $0<r_k<\dist(x_0,\Om\cap\bd A_k)$ be a sequence decreasing to zero.

Then for each $k=1,2,\ldots$ there is a component $G_{j_k}(r_k)$ of $B(x_0, r_k) \cap \Om$ intersecting $A_l$ for each $l=1,2,\ldots$, and such that $x_0\in \overline{G_{j_k}(r_k)}$ and $G_{j_k}(r_k) \subset A_k$.

Moreover, there exists a prime end $[F_k]$  such that  $I[F_k]=\{x_0\}$,
$F_k=G_{j_k}(r_k)$ for some
$1\le j_k\le N(r_k)$ and
$F_{k}\subset A_k$, $k=1,2,\ldots$.

\end{lem}

 In the proof below we follow the approach of the proof for Theorem 7.4 in Soderborg~\cite{sod}. However, we employ the prime ends as described in Section~\ref{sec-pet}, especially the structure of the prime ends boundary for domains finitely connected at the boundary.

 \begin{proof}[Proof of Theorem~\ref{thm-Royden-prime}]
  Recall that by Theorem~\ref{thm-fin-con-homeo} every $x\in\bd \Om$ is an impression of some prime end $[E_n]\in \bdyP \Om$ and all prime ends are singletons, i.e. $\bdyP \Om=\bdySP \Om$. We define a map $R:\bdyP \Om\to \Delta$ as follows:
   \[
   R([E_n])=\Psi_x,
   \]
  for $\Psi_x$ satisfying the following two conditions:
  \begin{enumerate}
  \item $\pi(\Psi_x)=I[E_n]=\{x\}\subset \bd \Om$, where $\pi:\Om^*\to \overline{\Om}$ is a natural continuous projection map, surjective from $\Delta$ onto $\bd \Om$ (see Section 4 and Theorems 4.2 and 4.3 in \cite{sod}).
  \item for every $\chi\in\Psi_x$ and for every neighborhood $U$ of $x$ (in $\R^n$) it holds that
  \[
   Q(U,\chi)=G_{j_k}(r_k),\qquad \hbox{ for some }j_k\in \N, r_k>0
  \]
   where $Q(U, \chi)$ is a unique component of $U\cap \Om$ such that each Royden net corresponding to $\chi$ lies eventually in $Q(U, \chi)$ (the existence of such a neighborhood is proved in \cite[Lemma 7.1]{sod}) and $G_{j_k}(r_k)$ is one of the sets constructed in Lemma~\ref{lem10-5}.
   \end{enumerate}

 Theorem 7.2 in \cite{sod} stays that two elements $\chi, \eta \in \Phi_x$ belong to the same component of $\Phi_x$ if and only if $Q(U, \chi)=Q(U, \eta)$ for each neighborhood $U$ of $x$ in $\R^n$. Moreover, equivalent chains of the given prime end have the same impressions. Therefore, the map $R$ is well defined and injective
 (by \cite[Theorem 7.2]{sod} and Theorem~\ref{thm-fin-con-homeo}). The proof will be completed once we show that $R$ is onto. Let $\Psi$ be a component of $\Psi_x$ and $\chi \in \Psi$. Similarly to construction in Lemma~\ref{cor-loc-conn} and Theorem~\ref{thm-fin-con-homeo} we consider a sequence of balls centered at $x$ and related sequence of open connected subsets of $\R^n$
 \begin{align*}
  &Q_j:=Q(B(x, 1/j), \chi) \\
  &Q_{j+1}\subset Q_j \qquad \hbox{for } j=1,2,\ldots.
 \end{align*}
 Let us choose a sequence of points $x_j\in Q_j$ for $j=1,2,\ldots$ and join every $x_j$ with $x_{j+1}$ by a path $\ga_j$. The resulting path $\ga$ has one endpoint at $x_1$ and $\ga(x_j)\to x$ for $j\to \infty$. Next, Lemma~\ref{lem10-5} provides us with the construction of prime end $[F_k]$ with $F_k=G_{j_k}(r_k)$ for $k=1,2,\ldots$ and $I[F_k]=\{x\}$. This together with the definition of $\ga$ imply that $[F_k]$ is accessible through $\ga$.

 Finally, for any neighborhood $U$ of $x$ it holds that $Q_m\subset U$ for $m$ large enough. Moreover, $Q(U, \chi)$ is a component of $U$ containing $Q_m$. On the other hand, $Q_m$ contains some subpath of $\ga$, and hence one of the acceptable sets $G_{j_l}(r_l)$ of prime end $[F_k]$ is a component of $U\cap \Om$ containing $Q_m$. In a consequence
 $Q(U,\chi)=G_{j_l}(r_l)$ and the proof is completed.
\end{proof}

The following observation is an immediate consequence of Theorem
\ref{thm51-sod} and the proof of Theorem~\ref{thm-Royden-prime}.

\begin{cor}
 Let $\Om\subset \R^n$ be a domain finitely connected at the boundary
and $f:\Om\to\Om'$ be a quasiconformal map. If $f$ extends
homeomorphically to a map $F:\overline{\Om}\to \overline{\Om'}$,
then there exists a correspondence between prime ends with
impression at a point $x\in \bd \Om$ and prime ends with impression
at point $F(x)\in \bd \Om'$.
\end{cor}
 In other words, prime ends associated with a given boundary point in $\bd \Om$ may not be mapped to prime ends with impressions at different points in $\bd \Om'$.

 Moreover, note that in the above corollary, the existence of a homeomorphic extension to topological closures of domains implies that $\Om'$ is finitely connected at the boundary.

\begin{proof}
 We utilize map $R$ studied in the proof of Theorem~\ref{thm-Royden-prime} and define a mapping $F_{\rm P}:\bdyP \Om \to \bdyP \Om'$ by the following formula
\[
  F_{\rm P}([E_n])=R_{\Om'}^{-1}\circ T^*\circ R_{\Om}([E_n]),\quad \hbox{for any }\, [E_n]\in \bdyP \Om.
\]
 By the above discussion $T^*$ is a homeomorphism (an extension of
$f$) between Royden compactifications $\Om^{*}$ and $(\Om')^*$ and
since mappings $R_\Om$ and $R_{\Om'}$ are bijections, the proof of
the corollary is completed.
\end{proof}


\end{document}